\documentclass[11pt]{article}
\usepackage{amsthm}
\usepackage{enumerate}
\usepackage{color}
\usepackage{mathtools}
\usepackage{amsmath}
\usepackage{amssymb}
\usepackage{amsfonts}
\usepackage[hyperfootnotes=false]{hyperref}

\theoremstyle{plain}
\newtheorem{theorem}{Theorem}[section]
\newtheorem{corollary}{Corollary}[section]
\newtheorem{lemma}{Lemma}[section]

\theoremstyle{definition}
\newtheorem{definition}{Definition}[section]
\newtheorem{example}{Example}[section]
\numberwithin{equation}{section}

\let\Item\item
\newenvironment{romanlist}{%

	\let\item\Item
	\begin{enumerate}
	}{%
	\end{enumerate}}

\begin{document}

\centerline{{\huge Chemical Reaction Systems with a Homoclinic}}

\medskip

\centerline{{\huge Bifurcation: an Inverse Problem}}
 
\medskip
\bigskip

\centerline{
\renewcommand{\thefootnote}{$*$}
{\Large Tomislav Plesa\footnote{
Mathematical Institute, University of Oxford, Radcliffe Observatory 
Quarter, Woodstock Road, Oxford, OX2 6GG, United Kingdom;
e-mails: tomislav.plesa@some.ox.ac.uk, erban@maths.ox.ac.uk}
\qquad 
\renewcommand{\thefootnote}{$\dagger$}
Tom\'{a}\v{s} Vejchodsk\'{y}\footnote{
Institute of Mathematics, Czech Academy of Sciences, 
\v{Z}itn\'{a} 25,
Praha 1, 115 67, Czech Republic, e-mail: vejchod@math.cas.cz}
\qquad
Radek Erban$^*$}}

\medskip
\bigskip

\noindent
{\bf Abstract}:
An inverse problem framework for constructing reaction systems with prescribed 
properties is presented. Kinetic transformations are defined and analysed as 
a part of the framework, allowing an arbitrary polynomial ordinary differential 
equation to be mapped to the one that can be represented as a reaction network. 
The framework is used for construction of specific two- and three-dimensional 
bistable reaction systems undergoing a supercritical homoclinic bifurcation, 
and the topology of their phase spaces is discussed.

\section{Introduction}
Chemical reaction networks under the mass action kinetics are relevant 
for both pure and applied mathematics. The time evolution of the
concentrations of chemical species is described by kinetic equations 
which are a subset of first-order, autonomous, ordinary differential 
equations (ODEs) with polynomial right-hand sides (RHSs). On the one 
hand, the kinetic equations define a canonical form for analytic 
ODEs, thus being important for pure mathematics~\cite{UNI1,UNI2}. 
They can display not only the chemically regular phenomenon of having 
a globally stable fixed point, but also the chemically exotic phenomena 
(multistability,  limit cycles, chaos). It is then no surprise that 
chemical reaction networks can perform the same computations as other 
types of physical networks, such as electronic and neural 
networks~\cite{Synth1}. On the other hand, reaction networks are 
a versatile modelling tool, decomposing processes from applications 
into a set of simpler elementary steps (reactions). The exotic phenomena 
in systems biology often execute specific biological functions, example 
being the correspondence between limit cycles and biological 
clocks~\cite{Intro11,Intro2}.

The construction of reaction networks displaying prescribed properties may 
be seen as an inverse problem in formal reaction kinetics~\cite{Toth2}, 
where, given a set of properties, a set of compatible reaction networks 
is searched for. Such constructions are useful in application areas such 
systems biology (as caricature models), synthetic biology (as blueprints), 
and numerical analysis (as test problems)~\cite{Intro1,Repress}. 
In systems biology, kinetic ODEs often have higher nonlinearity degree and 
higher dimension, thus not being easily amenable to mathematical 
analysis. Having ODEs with lower nonlinearity degree and lower 
dimension allows for a more detailed mathematical analysis, and also 
adds to the set of test problems for numerical methods designed for 
more challenging real-world problems. In synthetic biology, such 
constructed systems may be used as a blueprint for engineering 
artificial networks~\cite{Repress}.

A crucial property of the kinetic equations is a lack of so-called 
cross-negative terms~\cite{Toth1}, corresponding to processes
that involve consumption of a species when its concentration is 
zero. Such terms are not directly describable by reactions, and may 
lead to negative values of concentrations. The existence of cross-negative 
terms, together with a requirement that the dependent variables are 
always finite, imply that not every nonnegative polynomial ODE 
system is kinetic, and, thus, further constrain the possible 
dynamics, playing an important role in the construction of reaction 
systems, chemical chaos, and pattern formation via Turing 
instabilities~\cite{Toth3,Toth4}. A trivial example of 
an ODE with a cross-negative term is given by 
$\mathrm{d} x/\mathrm{d} t = - k$, for constant $k > 0$, where the 
term $-k$, although a polynomial of degree zero, nevertheless 
cannot be directly represented by a reaction, and results in $x < 0$. 

In two dimensions, where the phase plane diagram allows for 
an intuitive reasoning, the exotic dynamics of ODE systems 
reduces to limit cycles and multistability. While two-dimensional 
nonkinetic polynomial ODE systems exhibiting a variety of 
such dynamics can be easily found in the literature, the same 
is not true for the more constrained two-dimensional kinetic 
ODE systems. Motivated by this, this paper consists of two 
main results: firstly, building upon the framework 
from~\cite{Toth2,Toth1}, an inverse problem framework 
suitable for constructing the reaction systems is presented in 
Section~\ref{sec:inverse}, with the focus on the so-called 
kinetic transformations, allowing one to map a nonkinetic into 
a kinetic system. Secondly, in Section~\ref{sec:construction}, 
the framework is used for construction of specific two- and 
three-dimensional bistable kinetic systems undergoing a global 
bifurcation known as a supercritical homoclinic bifurcation. 
The corresponding phase planes contain a stable limit cycle 
and a stable fixed point, with a parameter controlling the 
distance between them, and their topology is discussed. Definitions 
and basic results regarding reaction systems are presented 
in Section~\ref{sec:notation}. A summary of the paper is presented 
in Section~\ref{sec:end}.

\section{Notation and definitions} 
\label{sec:notation}

The notation and definitions in this paper are 
inspired by~\cite{Feinberg,Craciun,Toth1}.

\begin{definition} \label{def:notation}
Let $\mathbb{R}$ be the space of real numbers, $\mathbb{R}_{\ge}$ the 
space of nonnegative real numbers, $\mathbb{R}_{>}$ the space 
of positive real numbers and $\mathbb{N} = \{0, 1, 2, 3, \dots \}$
the set of natural numbers. Given a finite set $\mathcal{S}$, with 
cardinality $|\mathcal{S}| = S$, the real space of formal sums 
$c = \sum_{s \in \mathcal{S}} c_s s$ is denoted by 
$\mathbb{R}^{\mathcal{S}}$ if $c_s \in \mathbb{R}$ for all 
$s \in \mathcal{S}$. It is denoted by $\mathbb{R}_{\ge}^{\mathcal{S}}$ 
if $c_s \in \mathbb{R}_{\ge}$ for all $s \in \mathcal{S}$;
by $\mathbb{R}_{>}^{\mathcal{S}}$ if $c_s \in \mathbb{R}_{>}$
for all $s \in \mathcal{S}$; and by  
$\mathbb{N}^{\mathcal{S}}$ if $c_s \in \mathbb{N}$
for all $s \in \mathcal{S}$; where the number $c_s$ is called 
the $s$-component of $c$ for $s \in \mathcal{S}$. \emph{Support} 
of $c \in \mathbb{R}^{\mathcal{S}}$ is defined as 
$\text{supp}(c) = \{s \in \mathcal{S}: c_s \ne 0\}$. 
Complement of a set $\mathcal{M} \subset \mathcal{S}$ 
is denoted by $\mathcal{M}^{\mathsf{c}}$, and given 
by $\mathcal{M}^{\mathsf{c}} = \mathcal{S} \setminus \mathcal{M}$. 
\end{definition}
\noindent
The formal sum notation is introduced so that unnecessary ordering 
of elements of a set can be avoided, such as when general frameworks 
involving sets are described, and when objects under consideration 
are vector components with irrelevant ordering. The usual vector 
notation is used when objects under consideration are equations 
in matrix form, and is put into using starting with 
equation~(\ref{eqn:nkv}).

\subsection{Reaction networks and reaction systems} 
\label{sec:rn}

\begin{definition} \label{def:rn}
A \emph{reaction network} is a triple 
$\{\mathcal{S}, \mathcal{C}, \mathcal{R}\}$, where
\begin{enumerate}[(i)]
\item $\mathcal{S}$ is a finite set, with elements $s \in \mathcal{S}$ 
called the \emph{species} of the network.
\item $\mathcal{C} \subset \mathbb{N}^{\mathcal{S}}$ is 
a finite set, with elements $c \in \mathcal{C}$, called 
the \emph{complexes} of the network,  such that 
$\bigcup _{c \in \mathcal{C}} \text{supp}(c) = \mathcal{S}$. 
Components of $c$ are called the \emph{stoichiometric coefficients}.
\item $\mathcal{R} \subset \mathcal{C} \times \mathcal{C}$ is a binary 
relation with elements $r = (c,c')$, denoted $r = c \to c'$, with 
the following properties:
\begin{enumerate}
\item $\forall c \in \mathcal{C}$
      $(c \to c) \notin \mathcal{R}$;
\item $\forall c \in \mathcal{C}$ $\exists c' \in \mathcal{C}$ such 
that $(c \to c') \in \mathcal{R}$ or $(c' \to c) \in \mathcal{R}$.
\end{enumerate}
Elements $r = c \to c'$ are called \emph{reactions} of the network, 
and it is said that $c$ \emph{reacts to} $c'$, with $c$ being called 
the \emph{reactant complex}, and $c'$ the \emph{product complex}. 
The \emph{order} of reaction $r$ is given by 
$o_r = \sum_{s \in \mathcal{S}} c_s < \infty$ for
$r = c \to c' \in \mathcal{R}$.
\end{enumerate}
\end{definition}
\noindent Note that as set $\mathcal{R}$ implies sets $\mathcal{S}$ 
and $\mathcal{C}$, reaction networks are denoted with $\mathcal{R}$, 
for brevity. Also, as it is unlikely that a reaction between more 
than three reactants occurs~\cite{Toth1}, in this paper we consider reactions with 
$o_r \le 3$. To represent some of the non-chemical processes as quasireactions, 
the \emph{zero complex} is introduced, denoted with $\varnothing$, 
with the property that $\text{supp}(\varnothing) = \emptyset$, 
where $\emptyset$ is the empty set. 

\begin{definition} \label{def:kinetics}
Let $\mathcal{R}$ 
be a reaction network and let 
$\kappa : \mathbb{R}_{\ge}^{\mathcal{S}} \to \mathbb{R}_{\ge}^{\mathcal{R}}$ 
be a continuous function which maps $x \in \mathbb{R}_{\ge}^{\mathcal{S}}$
(called ``species concentrations'') into 
$\kappa(x) \in \mathbb{R}_{\ge}^{\mathcal{R}}$ (called ``reaction rates'').
Then $\kappa$ is said to be a \emph{kinetics} for $\mathcal{R}$
provided that, for all $x \in \mathbb{R}_{\ge}^{\mathcal{S}}$
and for all $r = (c \to c') \in \mathcal{R}$, positivity
$\kappa_{r}(x) > 0$ is satisfied if and only if 
$\mathrm{supp}(c) \subset \mathrm{supp}(x)$. 
\end{definition}

\noindent
An interpretation of Definition~\ref{def:kinetics} is that a reaction, 
to which a kinetics can be associated, can occur if and only if all 
the reactant species concentrations are nonzero. 

\begin{definition} 
A reaction network $\mathcal{R}$ augmented with a kinetics $\kappa$ 
is called a \emph{reaction system}, and is denoted 
$\{\mathcal{R}, \kappa\}$.
\end{definition} 

\begin{definition} \label{def:kinfun}
Given a reaction system $\{\mathcal{R}, \kappa\}$, the induced 
\emph{kinetic function}, $\mathcal{K}(\cdot ; \, \mathcal{R}) :
\mathbb{R}_{\ge}^{\mathcal{S}} \to \mathbb{R}^{\mathcal{S}}$, is 
given by $\mathcal{K}(x ; \, \mathcal{R}) = \sum_{r \in \mathcal{R}}
\kappa_{r}(x) (c' - c)$ where $r = c \to c'$. The induced 
\emph{system of kinetic equations}, describing the time evolution 
of species concentrations $x \in \mathbb{R}_{\ge}^{\mathcal{S}}$, 
takes the form of a system of autonomous first-order ordinary 
differential equations (ODEs), and is given by
\begin{align}
\frac{\mathrm{d} x}{\mathrm{d}  t} & 
=  \mathcal{K}(x ; \, \mathcal{R}). \label{eqn:kinetic}
\end{align}
\end{definition}
\noindent Note that the kinetic function uniquely defines the system 
of kinetic equations, and vice-versa. In this paper, the species 
concentrations satisfying equation~(\ref{eqn:kinetic}) are required 
to be finite, i.e. $x_s < \infty$, for $s \in \mathcal{S}$, and for 
$t \ge 0$, except possibly for initial conditions located on 
a finite number of $(S-z)$-dimensional subspaces of 
$\mathbb{R}_{\ge}^{\mathcal{S}}$, $z \ge 1$.

\begin{definition} \label{def:mak}
Kinetics $\kappa$ is called the \emph{mass action kinetics} if 
$\kappa_{r}(x) = k_{r} x^{c}$, for $r = (c \to c') \in \mathcal{R}$, 
where $k_{r} > 0$ is the \emph{rate constant} of reaction $r$, 
and $x^{c} = \prod_{s \in S} x_s^{c_{s}}$, with $0^0 = 1$. A reaction 
system with the mass action kinetics is denoted $\{\mathcal{R}, k\}$, 
and the corresponding kinetic function is denoted 
$\mathcal{K}(x; \, k) \equiv \mathcal{K}(x; \, \mathcal{R}) 
= \sum_{r \in \mathcal{R}} k_{r} (c' - c) x^{c}$,
where $k \in \mathbb{R}_{>}^{\mathcal{R}}$.
\end{definition}
\noindent A review of the mass action kinetics can 
be found in~\cite{MAK}. In this paper, most of the results 
are stated with kinetics fixed to the mass action kinetics. 
This is not restrictive, as an arbitrary analytic function 
can always be reduced to a polynomial one~\cite{UNI1}. 

\begin{example} \label{ex:rn} Consider the following reaction 
network (consisting of one reaction) under the mass action 
kinetics: 
\begin{align}
r_1: \;  & \; \; s_1 + s_2 \xrightarrow[]{ k_{1} }  2 s_2, 
\label{eqn:exrn1} 
\end{align}
so that $\mathcal{S} = \{s_1,s_2\}$, $\mathcal{C} = \{ s_1 + s_2, 2 s_2\}$,
$\mathcal{R} = \{ s_1 + s_2 \to  2 s_2 \}$ and $k = \{k_{1}\}$. 
Concentration $x \in \mathbb{R}_{\ge}^{\mathcal{S}}$
has two components. To simplify our notation, we write
$x_1 = x_{s_1}$, $x_2 = x_{s_2}$, and 
$\mathcal{K}_{1}(x ; \, k) = \mathcal{K}_{s_1}(x ; \, \mathcal{R})$,  
$\mathcal{K}_{2}(x ; \, k) = \mathcal{K}_{s_2}(x ; \, \mathcal{R})$. 
Then the induced system of kinetic equations is given by
\begin{align}
\frac{\mathrm{d} x_{1}}{\mathrm{d} t}  \; \; 
= \mathcal{K}_{1}(x; \, k) & 
= - k_{1} x_{1} x_{2}, \label{eqn:exrn2a}\\
\frac{\mathrm{d} x_{2}}{\mathrm{d} t}  \; \; 
= \mathcal{K}_{2}(x; \,k) & 
=  k_{1} x_{1} x_{2}. \label{eqn:exrn2b} 
\end{align}
\end{example}

\subsection{Kinetic and nonkinetic functions}
In this subsection, nonkinetic functions are defined, and 
further notation for kinetic and nonkinetic functions 
taking the mass action form is presented.

\begin{definition} \label{def:cnt}
Let $f: \mathbb{R}^{\mathcal{S}}_{\ge} \to \mathbb{R}^{\mathcal{S}}$
be given by $f_{s}(x) = \sum_{r \in \mathcal{R}} f_{s r}(x)$,
where $f_{s r}(x) \in \mathbb{R}$, 
for $x \in \mathbb{R}^{\mathcal{S}}_{\ge}$,
$s \in \mathcal{S}$ and $r \in \mathcal{R}.$ 
If $\exists s \in \mathcal{S}$, $\exists r \in \mathcal{R}$ 
and $\exists x \in \mathbb{R}^{\mathcal{S}}_{\ge}$ such that 
$s \in \mathrm{supp}^{\mathsf{c}}(x)$ and $f_{s r}(x) < 0$, 
then $f_{s r}(x)$ is called a \emph{cross-negative term}, 
and function $f(x)$ and ODE system $\mathrm{d} x/\mathrm{d} t = f(x)$ 
are said to be \emph{nonkinetic}.
\end{definition}
\noindent
An interpretation of a cross-negative term is that the process 
corresponding to such a term would consume at least one reactant 
even when its concentration is zero, so that it cannot be 
represented as kinetic reactions. 

Kinetic and nonkinetic functions taking the mass action kinetics 
form are central to this paper. The related notation is introduced 
in the following definition. 
\begin{definition} \label{def:funspaces}
Let $\mathcal{P}(\cdot ; \, k):\mathbb{R}^{\mathcal{S}} \to \mathbb{R}^{\mathcal{S}}$, 
$k \in \mathbb{R}^{\mathcal{R}}$, be a polynomial function with 
polynomial degree $\mathrm{deg}(\mathcal{P}(x; \, k)) \le m$, $m \in \mathbb{N}$. 
Then, the set of functions $\mathcal{P}(x ; \, k)$ is denoted by 
$\mathbb{P}_{m}(\mathbb{R}^{\mathcal{S}}; \, \mathbb{R}^{\mathcal{S}})$. 
If  $\mathcal{P}(x ; \, k)$ is a kinetic function, it is denoted by 
$\mathcal{K}(x; \, k)$, $k \in \mathbb{R}_{>}^{\mathcal{R}}$, and the 
set of such functions is denoted by 
$\mathbb{P}_{m}^{\mathcal{K}}(\mathbb{R}_{\ge}^{\mathcal{S}}; \, 
\mathbb{R}^{\mathcal{S}})$. If $\mathcal{P}(x ; \, k)$ is a nonkinetic 
function, it is denoted by $\mathcal{N}(x ; \, k)$, 
$k \in \mathbb{R}^{\mathcal{R}}$, and the set of such functions 
with domain $\mathbb{R}^{\mathcal{S}}$ is denoted 
by $\mathbb{P}_{m}^{\mathcal{N}}(\mathbb{R}^{\mathcal{S}}; \, 
\mathbb{R}^{\mathcal{S}})$, while with domain $\mathbb{R}_{\ge}^{\mathcal{S}}$ by $\mathbb{P}_{m}^{\mathcal{N}}(\mathbb{R}_{\ge}^{\mathcal{S}}; \, \mathbb{R}^{\mathcal{S}})$. 
\end{definition}
\noindent
Note that a system $\{R, k\}$, corresponding to $\mathcal{N}(x ; \, k)$ in 
Definition~\ref{def:funspaces}, has a well-defined reaction network 
$R$ (for $r = c \to c'$, $r \in \mathcal{R}$, we restrict $c, c'$ to positive 
integers), but an ill-defined kinetics taking the mass action 
\emph{form} (we allow set $k$ to have elements that are negative). 
Thus, set $k$ corresponding to $\mathcal{N}(x ; \, k)$ cannot be 
interpreted as a set of reaction rate constants, as opposed to 
set $k$ corresponding to $\mathcal{K}(x ; \, k)$ (see also Example~\ref{ex:cnt}). 
Note also that 
$\mathbb{P}_{m}(\mathbb{R}_{\ge}^{\mathcal{S}}; \, \mathbb{R}^{\mathcal{S}}) = \mathbb{P}_{m}^{\mathcal{K}}(\mathbb{R}_{\ge}^{\mathcal{S}}; \, \mathbb{R}^{\mathcal{S}}) \cup \mathbb{P}_{m}^{\mathcal{N}}(\mathbb{R}_{\ge}^{\mathcal{S}}; \, \mathbb{R}^{\mathcal{S}})$, with $\mathbb{P}_{m}^{\mathcal{K}}(\mathbb{R}_{\ge}^{\mathcal{S}}; \, \mathbb{R}^{\mathcal{S}}) \cap \mathbb{P}_{m}^{\mathcal{N}}(\mathbb{R}_{\ge}^{\mathcal{S}}; \, \mathbb{R}^{\mathcal{S}}) = \emptyset$.

\subsection{Properties of kinetic functions}
From Definition~\ref{def:kinetics} it follows that a kinetic function 
$\mathcal{K}(x ; \, \mathcal{R})$ has a structural property: cross-negative 
terms are absent. In this subsection, further properties of 
$\mathcal{K}(x ; \, \mathcal{R})$ are defined: nonnegativity 
(absence of cross-negative effect), and a structural property called
$x$-factorability.
\begin{definition} \label{def:cne}
Let $f: \mathbb{R}^{\mathcal{S}}_{\ge} \to \mathbb{R}^{\mathcal{S}}$
be given by $f_{s}(x) = \sum_{r \in \mathcal{R}} f_{s r}(x)$,
where $f_{s r}(x) \in \mathbb{R}$, 
for $x \in \mathbb{R}^{\mathcal{S}}_{\ge}$,
$s \in \mathcal{S}$ and $r \in \mathcal{R}.$ 
If $\forall s \in \mathcal{S}$, 
$\forall x \in \mathbb{R}^{\mathcal{S}}_{\ge}$, 
$s \in \mathrm{supp}^{\mathsf{c}}(x) \Rightarrow f_s(x) \ge 0$, 
then $f(x)$ and $\mathrm{d} x/\mathrm{d} t = f(x)$ are said to 
be \emph{nonnegative}. Otherwise, $f(x)$ and 
$\mathrm{d} x/\mathrm{d} t = f(x)$ are said to be \emph{negative}, 
and a \emph{cross-negative effect} is said to 
exists $\forall x \in \mathbb{R}^{\mathcal{S}}_{\ge}$ for which 
$\exists s \in {\mathcal{S}}$ such that $s \in \mathrm{supp}^{\mathsf{c}}(x)$ and
$f_s(x) < 0$. 
\end{definition}
\noindent Note that the absence of cross-negative terms implies 
nonnegativity, but the converse is not necessarily
true~\cite{Toth1,ChemDyn1}, i.e. an ODE system may have 
cross-negative terms, without having a cross-negative effect, as 
we will show in Example~\ref{ex:cnt}.

Cross-negative terms play an important role in mathematical constructions 
of reaction systems, in the context of chaos in kinetic equations, 
and pattern formation via Turing instabilities~\cite{Toth3,Toth4}. 
In the context of oscillations, as a generalization of the result
in~\cite{Constr1}, one can prove that in two-dimensional reaction 
systems with mass action form and with at most bimolecular reactions, 
the nonexistence of a cross-negative effect in the ODEs 
is a sufficient condition for nonexistence of limit cycles 
(see~\ref{OsciCNT}).

\begin{example} \label{ex:cnt} Consider the following ODE system
with polynomial RHS: 
\begin{align}
\frac{\mathrm{d} x_{1}}{\mathrm{d} t}  
= \mathcal{P}_{1}(x; \, k) & 
= 1 + x_1^2 + 2 k x_2 + x_2^2, 
\label{eqn:excnt1a}
\\
\frac{\mathrm{d} x_{2}}{\mathrm{d} t} 
= \mathcal{P}_{2}(x; \,k) & 
= 1,
\label{eqn:excnt1b}
\end{align}
where $\mathcal{P}(x; \, k) \in \mathbb{P}_{2}(\mathbb{R}^{\mathcal{S}}; 
\, \mathbb{R}^{\mathcal{S}})$, $S = 2$, 
$k \in \mathbb{R}$ and $x = \{x_1, x_2\}$. 
Considering $x_1 = 0$ and $x_2 > 0$, it follows that 
$\mathcal{P}_1(\{0, x_2\}; \, k) 
= 1 + 2 k x_2 + x_2^2$. Then:
\begin{enumerate}[(i)]
\item If $k \ge 0$, 
then~(\ref{eqn:excnt1a})--(\ref{eqn:excnt1b}) contains no 
cross-negative terms, and so it is kinetic: $\mathcal{P}(x; \, k)  
\in \mathbb{P}_{2}^{\mathcal{K}}(\mathbb{R}_{\ge}^{\mathcal{S}}; 
\, \mathbb{R}^{\mathcal{S}})$. 
\item If $k < 0$, then~(\ref{eqn:excnt1a})--(\ref{eqn:excnt1b}) 
contains one cross-negative term, $2 k x_2$, and so it is nonkinetic: 
$\mathcal{P}(x; \, k)  
\in \mathbb{P}_{2}^{\mathcal{N}}(\mathbb{R}^{\mathcal{S}}; \, \mathbb{R}^{\mathcal{S}})$.
 \begin{enumerate}
\item If $-1 \le k < 0$, then~(\ref{eqn:excnt1a})--(\ref{eqn:excnt1b}) 
contains no cross-negative effect, and so it is nonnegative.
\item If $k < -1$, then~(\ref{eqn:excnt1a})--(\ref{eqn:excnt1b}) 
contains a cross-negative effect for $x=\{0,x_2\}$, where
$x_2 \in \big(- k - \sqrt{k^2 - 1}, - k + \sqrt{k^2 - 1 } \big)$, 
and so it is negative.
\end{enumerate}
\end{enumerate}
System~(\ref{eqn:excnt1a})--(\ref{eqn:excnt1b}) induces a reaction system 
only in case (i). In particular,
nonnegative ODE system~(\ref{eqn:excnt1a})--(\ref{eqn:excnt1b}) 
with $\mathcal{P}(x; \, k)  
\in \mathbb{P}_{2}^{\mathcal{N}}(\mathbb{R}_{\ge}^{\mathcal{S}}; \, 
\mathbb{R}^{\mathcal{S}})$ in part (ii)(a) does \emph{not} induce 
a reaction system (although, given a nonnegative initial 
condition, the solution of~(\ref{eqn:excnt1a})--(\ref{eqn:excnt1b})
is nonnegative for all forward times).
\end{example}

\begin{definition} \label{def:xf}
Let $f: \mathbb{R}^{\mathcal{S}}_{\ge} \to \mathbb{R}^{\mathcal{S}}$
be given by $f_{s}(x) = \sum_{r \in \mathcal{R}} f_{s r}(x)$,
where $f_{s r}(x) \in \mathbb{R}$, 
for $x \in \mathbb{R}^{\mathcal{S}}_{\ge}$,
$s \in \mathcal{S}$ and $r \in \mathcal{R}.$ 
Then term $f_{s r}(x)$ is said to be \emph{$x_s$-factorable} if 
$f_{s r}(x) = x_s p_{s r}(x)$, where $p_{s r}(x)$ is a polynomial 
function of $x$. If $\exists s \in \mathcal{S}$, such that 
$f_s(x) = x_s \sum_{r \in \mathcal{R}} p_{s r}(x)$, 
then $f(x)$ and ODE system $\mathrm{d} x/\mathrm{d} t = f(x)$ are 
said to be \emph{$x_s$-factorable}. 
If $\forall s \in \mathcal{S}$ it is true that  
$f_s(x) = x_s \sum_{r \in \mathcal{R}} p_{s r}(x)$, then 
$f(x)$ and ODE system $\mathrm{d} x/\mathrm{d} t = f(x)$ 
are said to be \emph{$x$-factorable}.
\end{definition}
\begin{example} \label{ex:xf} System~(\ref{eqn:exrn2a})--(\ref{eqn:exrn2b}) 
is $x$-factorable, 
since $\mathcal{K}_{1}(x; \, k) = x_1 (-k_1 x_2)$ 
and $\mathcal{K}_{2}(x; \, k)= x_2 (k_1 x_1)$.
\end{example}

\noindent 
$X$-factorable ODE systems are a subset of kinetic equations under 
the mass action kinetics~\cite{Xfactor} (see also Section~\ref{sec:xft}).

\section{Inverse problem for reaction systems} 
\label{sec:inverse}
In some applications, we are interested in the \emph{direct problem}: 
we are given a reaction network with kinetics, i.e. a reaction system 
$\{\mathcal{R}, \kappa\}$, and we then analyse the
induced system of kinetic equations (\ref{eqn:kinetic}) in order to 
determine properties of the reaction system. For example, an output 
of a direct problem might consist of verifying that the kinetic 
equations undergo a bifurcation.
In this paper, we are interested in the \emph{inverse problem}: we 
are given a property of an unknown reaction system, and we would 
then like to construct a reaction system displaying the property. 
The inverse problem framework described in this section is 
inspired by~\cite{Toth1,Toth2}.

The first step in the inverse problem is, given a quantity that depends 
on a kinetic function, to find a compatible 
kinetic function $\mathcal{K}(x ; \, \mathcal{R})$, while the second step 
is then to find a reaction system $\{\mathcal{R}, \kappa\}$ induced by the 
kinetic function. The 
second step is discussed in more detail in Section~\ref{sec:cn}, 
while the first step in Section~\ref{sec:kt}.
The constructions of a reaction system $\{\mathcal{R}, \kappa\}$ 
often involve constraints defining simplicity of the system 
(e.g. see~\cite{Constr3}), and the simplicity can be related 
to the kinetic equations (structure and dimension of the equations, 
and/or the phase space), and/or to reaction networks (cardinality, 
conservability, reversibility, deficiency). How the simplicity 
constraints are prioritized depends on the application area, with 
simplicity of the kinetic equations being more important for 
mathematical analysis, while simplicity of the reaction networks 
for synthetic biology. 

\subsection{The canonical reaction network} 
\label{sec:cn}

Let us assume that we are able to construct an ODE system
of the form (\ref{eqn:kinetic}) where its RHS is a kinetic
function, $\mathcal{K}(x ; \, \mathcal{R})$, and the system 
has the property required by the inverse
problem. Then, one can always find a reaction system induced 
by the kinetic function~\cite{Toth3,Toth2}. 
While, for a fixed kinetics, a reaction network induces kinetic function 
uniquely by definition (see Definition~\ref{def:kinfun}), the 
converse is not true -- the inverse mapping of the kinetic function 
to the reaction networks is not unique --
a fact known as the fundamental dogma of chemical 
kinetics~\cite{Craciun,Gabor2,Toth3}. For example, 
in~\cite{Gabor2}, for a fixed kinetic function and a fixed 
set of complexes ($\mathcal{C}$ fixed), mixed integer programming 
is used for numerical computation of different induced reaction 
networks with varying properties. On the other hand, a constructive 
proof that every kinetic function induces a reaction system 
is given in~\cite{Toth3}, where $\mathcal{C}$ is generally not fixed 
(product complexes may be created), but the construction can be 
performed analytically, and it \emph{uniquely} defines an induced 
reaction system for a given kinetic function. The procedure is 
used in this paper, so it is now defined.

\begin{definition} \label{def:cn}
Let 
$\kappa : \mathbb{R}_{\ge}^{\mathcal{S}} \to \mathbb{R}_{\ge}^{\mathcal{R}}$
be a kinetics. Consider the kinetic function given by 
$\mathcal{K}(x ; \, \mathcal{R}) = 
\sum_{r \in \mathcal{R}} d_r \kappa_{r}(x)$,
where 
$x \in \mathbb{R}^{\mathcal{S}}_{\ge}$ and 
$d_r \in \mathbb{R}^{\mathcal{S}}$. Let us map 
$\mathcal{K}(x ; \, \mathcal{R})$ to a reaction 
system $\{\mathcal{R}_{\mathcal{K}^{-1}},\kappa_{\mathcal{K}^{-1}}\}$ 
with complexes and kinetics given by:
\begin{enumerate}[(i)]
\item Reactant complexes, $c_r$, are assumed to be uniquely obtainable from 
$\kappa_{r}(x)$ for $r \in \mathcal{R}$, which is true in the case of 
the mass action kinetics.
\item Reaction $c_r \to c_{rs}'$ is
then constructed for each $r \in \mathcal{R}$ and 
$s \in \mathcal{S}$, where new product complexes are given
by $c_{rs}' = c_r + \mathrm{sign}(d_{rs}) s$, with
$\mathrm{sign}(\cdot)$ being the sign function. 
\item The new kinetics is then defined as 
$\kappa_{\mathcal{K}^{-1} r'}(x) \equiv |d_{rs}| \kappa_{r}(x)$, 
for $r \in \mathcal{R}$, $s \in \mathcal{S}$, where 
$r' \in \mathcal{R}_{\mathcal{K}^{-1}}$.
\end{enumerate}
The induced reaction system 
$\{\mathcal{R}_{\mathcal{K}^{-1}},\kappa_{\mathcal{K}^{-1}}\}$ 
is called the canonical reaction system, with 
$\mathcal{R}_{\mathcal{K}^{-1}}$ being the 
\emph{canonical reaction network}.
\end{definition}

\noindent
Note that the procedure in Definition \ref{def:cn}
creates a reaction for each term in each kinetic
equation. Note also that each reaction leads to a change in copy number
of precisely one chemical species, and the change in the copy number is 
equal to one. Thus, the canonical reaction networks are simple in the 
sense that they can be constructed from a kinetic function in 
straightforward way, while they generally do not contain minimal 
number of reactions.

\begin{example} \label{ex:cn} The canonical reaction network 
for system~(\ref{eqn:exrn2a})--(\ref{eqn:exrn2b}) is given by
\begin{eqnarray}
\begin{aligned}
r_1: \;  & \; \; s_1 + s_2 \xrightarrow[]{ k_{1} }  s_2, \\
r_2: \;  & \; \; s_1 + s_2 \xrightarrow[]{ k_{2} }  s_1 + 2 s_2,\label{eqn:excn1}
\end{aligned}
\end{eqnarray}
so that $\mathcal{S} = \{s_1,s_2\}$, 
$\mathcal{C} = \{ s_1 + s_2, s_2, s_1 + 2 s_2\}$, 
$\mathcal{R}_{\mathcal{K}^{-1}} 
= \{ s_1 + s_2 \to  s_2, s_1 + s_2 \to  s_1 + 2 s_2 \}$ 
and $k_{\mathcal{K}^{-1}}  = \{k_{1},k_{2}\}$, $k_2 = k_1$. Note 
that the canonical reaction
network~(\ref{eqn:excn1}) contains more reactions than the 
original network~(\ref{eqn:exrn1}).
\end{example}

\subsection{Kinetic transformations} 
\label{sec:kt}

Firstly, mapping a solution-dependent quantity to the RHS of an ODE system 
is much more likely to result in nonkinetic 
functions,  $\mathcal{N}(x ; \, \mathcal{R})$, on the RHS
(see Definition~\ref{def:cnt})~\cite{Toth1}. 
However, only kinetic functions induce reaction networks, as 
exemplified in Example~\ref{ex:cnt}. Secondly, even if mapping 
a solution-dependent quantity results in a kinetic function, it may 
be necessary to modify the function in order to satisfy given constraints, 
and this may change the kinetic function into a nonkinetic function. 
For these two reasons, it is beneficial to study mappings that 
can transform arbitrary functions into kinetic functions. This 
motivates the following definition, for the case of mass action 
kinetics, that relies on the notation introduced 
in Definition~\ref{def:funspaces}.

\begin{definition} \label{def:kinetictransf}
Let $\mathcal{P}(x; \, k) \in  
\mathbb{P}_{m}(\mathbb{R}^{\mathcal{S}}; \, 
\mathbb{R}^{\mathcal{S}})$, $k \in \mathbb{R}^{\mathcal{R}}$, 
i.e. $\mathcal{P}(x; \,k)$ is a polynomial function.
Consider the corresponing ODE system in the formal sum notation 
\begin{align}
\frac{\mathrm{d} x}{\mathrm{d} t} = \mathcal{P}(x; \,k), \label{eqn:nk}
\end{align}
where $x \equiv x(t) \in \mathbb{R}^{\mathcal{S}}$. 
Then, a transformation $\Psi$ is called 
a \emph{kinetic transformation} if the following 
conditions are satisfied:
\begin{enumerate}[(i)]
\item $\Psi \!:\! \mathbb{P}_{m}(\mathbb{R}^{\mathcal{S}}; 
\, \mathbb{R}^{\mathcal{S}}) \to \mathbb{P}_{\bar{m}}^{\mathcal{K}}(\mathbb{R}_{\ge}^{\mathcal{\bar{S}}}; 
\, \mathbb{R}^{\mathcal{\bar{S}}})$, 
$\bar{m} \ge m$, $\bar{S} \ge S$, maps the polynomial function 
$\mathcal{P}(x; \, k)$ into a kinetic function 
$\mathcal{K}(\bar{x}; \, \bar{k}) \equiv 
(\Psi (\mathcal{P}) )(\bar{x}; \, \bar{k})$
for $\bar{x} \in \mathbb{R}_{\ge}^{\mathcal{\bar{S}}}$
and $\bar{k} \in \mathbb{R}_{>}^{\mathcal{\bar{R}}}$.
\item Let $x^{*}$ be a fixed point of (\ref{eqn:nk}) that 
is mapped by $\Psi$ to fixed point 
$\bar{x}^{*} \in \mathbb{R}_{\ge}^{\mathcal{\bar{S}}}$
of the system of kinetic equations (\ref{eqn:kinetic}) 
with $\mathcal{K}(\bar{x}; \, \bar{k})$ 
on its RHS. Let also the eigenvalues of the Jacobian matrix 
of $\mathcal{P}(x; \, k)$, $J(x^{*}; \, k)$, denoted by 
$\lambda_{n}$, $n = 1, 2, \ldots, S$, be mapped to 
the eigenvalues of Jacobian of $\mathcal{K}(\bar{x}; \, \bar{k})$, 
$J_{\Psi}(\bar{x}^{*}; \, \bar{k})$, which are denoted by 
$\bar{\lambda}_{n}$, $n = 1, 2, \ldots, S$. Then, for every such 
fixed point $x^{*}$ it must be true that $\mathrm{sign}(\lambda_{n}) 
= \mathrm{sign}(\bar{\lambda}_{n})$, $n = 1, 2, \ldots, S$, and, 
if there are any additional eigenvalues $\bar{\lambda}_{n}$, 
$n = S, S+1, \ldots, \bar{S}$, they must satisfy 
$\mathrm{sign}(\bar{\lambda}_{n}) < 0$.
\end{enumerate}
If any of the condition (i)--(ii) is not true, $\Psi$ is 
called a \emph{nonkinetic transformation}.
\end{definition}
\noindent
Put more simply, given an input polynomial function, a kinetic 
transformation must (i) map the input polynomial function into an 
output kinetic function, and (ii) the output function must be 
locally topologically equivalent to the input function in 
the neighbourhood of the corresponding fixed points, and the 
dynamics along any additional dimensions of the output 
function (corresponding to the additional species) must 
asymptotically tend to the corresponding fixed point. Let us 
note that the output function is defined only in the nonnegative 
orthant, so that the topological equivalence must hold only near 
the fixed points of the input function that are mapped to 
the nonnegative orthant under kinetic transformations.

One may wish to impose a set of constraints on an output function, 
such as requiring that a predefined region of interest in 
the phase space of the input function is mapped to the positive 
orthant of the corresponding output function. A subset of 
constraints is now defined.
\begin{definition} \label{def:constraints}
Let $\mathcal{P}(x; \, k) \in  \mathbb{P}_{m}(\mathbb{R}^{\mathcal{S}}; 
\, \mathbb{R}^{\mathcal{S}})$, $k \in \mathbb{R}^{\mathcal{R}}$. 
Let also $\phi_j : \mathbb{R}^{\mathcal{R}} \to \mathbb{R}$ 
be a continuous function, mapping set $k$ into 
$\phi_j(k) \in \mathbb{R}$, $j = 1, 2, \ldots, J$. Then, set 
$\Phi \equiv \{\phi_j(k) \ge 0 : j = 1, 2, \ldots, J\}$ is called 
a \emph{set of constraints}.
\end{definition}

There are two sets of kinetic transformations. The first, and 
the preferred, set of possible kinetic transformations are affine 
transformations, which are discussed in Section~\ref{sec:affine}. 
Affine transformations may be used, not only as possible kinetic 
transformations, but also to satisfy a set of constraints. 
The second set, necessarily used when affine transformations fail, 
are nonlinear transformations that replace cross-negative 
terms, with $x$-factorable 
terms (see Definition~\ref{def:xf}), without introducing new 
cross-negative terms, and two such transformations are discussed in 
Sections~\ref{sec:xft} and~\ref{sec:QSSA}. In choosing a nonlinear 
transformation, one generally chooses between obtaining, 
on the one hand, lower-dimensional kinetic functions with 
higher-degree of nonlinearity (i.e. lower $\bar{S}/S$ and higher 
$\bar{m}/m$ in Definition~\ref{def:kinetictransf}(i)) and/or 
higher numbers of the nonlinear terms, and, on the other hand,
higher-dimensional kinetic functions with lower degree of 
nonlinearity (i.e. higher $\bar{S}/S$ and lower $\bar{m}/m$)
and/or lower numbers of the nonlinear terms.

Before describing the transformations in a greater detail, the 
usual vector notation is introduced and related to the formal 
sum notation from Section~\ref{sec:notation}. The vector notation 
is used when ODE systems are considered in matrix form, while 
the formal sum notation is used when ODE systems are considered 
component-wise.

\smallskip

\noindent
\textbf{Notation}. Let $|\mathcal{S}| = S$, $|\mathcal{C}| = C$ 
and $|\mathcal{R}| = R$, and suppose $\mathcal{S}$, $\mathcal{C}$ 
and $\mathcal{R}$ are each given a fixed ordering with indices 
being $n = 1, 2, \ldots, S$, $i = 1, 2, \ldots, C$, and $l = 1, 2, \ldots, R$, 
respectively, i.e. one can identify the ordered components of 
formal sums with components of Euclidean vectors. Let also the 
indices $s_n$ be denoted by $n$, $n = 1, 2, \ldots, S$, for brevity. 
Then, the kinetic equations under the mass action kinetics in 
the formal sum notation are given by (\ref{eqn:kinetic}). In this
section, we start with equations which have more general polynomial, and not necessarily kinetic, 
functions on the RHS, i.e. the ODE system is written in the
formal sum notation as (\ref{eqn:nk}), 
while in the usual vector notation by
\begin{align}
\frac{\mathrm{d} \mathbf{x}}{\mathrm{d} t} & 
= \boldsymbol{\mathcal{P}}(\mathbf{x}; \, \mathbf{k}), \label{eqn:nkv}
\end{align}
where $\boldsymbol{\mathcal{P}}(\mathbf{x}; \, \mathbf{k}) \in \mathbb{P}_{m}(\mathbb{R}^{S}; \, \mathbb{R}^{S})$, $\mathbf{x} \in \mathbb{R}^{S}_{\ge}$, 
and $\mathbf{k} \in \mathbb{R}^{R}$. 

\subsubsection{Affine transformation} 
\label{sec:affine}

\begin{definition} \label{def:affine}
Consider applying an arbitrary nonsingular matrix 
$A \in \mathbb{R}^{S \times S}$ on equation~(\ref{eqn:nkv}), 
resulting in:
\begin{align}
\frac{\mathrm{d} \mathbf{\bar{x}}}{\mathrm{d} t} & = 
A \, \boldsymbol{\mathcal{P}}(A^{-1} \mathbf{\bar{x}}; \, \mathbf{\bar{k}})
\equiv
(\Psi_A \boldsymbol{\mathcal{P}})(\mathbf{\bar{x}}; \, \mathbf{\bar{k}}),
\label{eqn:affine} 
\end{align}
where $\mathbf{\bar{x}} = A \mathbf{x}$, and $\mathbf{\bar{k}}$ 
is a vector of new rate constants obtained from $\mathbf{k}$
by rewriting the polynomial on the RHS of (\ref{eqn:affine})
into the mass action form. Then $\Psi_A : \mathbb{P}_{m}(\mathbb{R}^{S}; \, \mathbb{R}^{S}) \to \mathbb{P}_{m}(\mathbb{R}^{S}; \, \mathbb{R}^{S})$, mapping $\boldsymbol{\mathcal{P}}(\mathbf{x}; \,\mathbf{k})$ to $(\Psi_A \boldsymbol{\mathcal{P}})(\mathbf{\bar{x}}; \, \mathbf{\bar{k}})$, is called a \emph{centroaffine transformation}.
If $A$ is an orthogonal matrix, then $\Psi_A$ 
is called an \emph{orthogonal transformation}.
\end{definition}

\begin{definition} \label{def:translation}
Consider substituting 
$\mathbf{\bar{x}} = \mathbf{x} + \boldsymbol{\mathcal{T}}$
in equation~(\ref{eqn:nkv}), where 
$\boldsymbol{\mathcal{T}} \in \mathbb{R}^{S}$, which results in:
\begin{align}
\frac{\mathrm{d} \mathbf{\bar{x}}}{\mathrm{d} t} & = 
\boldsymbol{\mathcal{P}}( \mathbf{\bar{x}} - \boldsymbol{\mathcal{T}}; \,
\mathbf{\bar{k}})
\equiv
(\Psi_{\mathcal{T}} \boldsymbol{\mathcal{P}})(\mathbf{\bar{x}}; \,
\mathbf{\bar{k}}),
\label{eqn:translation} 
\end{align}
where $\mathbf{\bar{k}}$ is a vector of the new rate constants
obtained from $\mathbf{k}$
by rewriting the polynomial on the RHS of (\ref{eqn:translation})
into the mass action form. Then $\Psi_{\mathcal{T}}: \mathbb{P}_{m}(\mathbb{R}^{S}; \, \mathbb{R}^{S}) \to \mathbb{P}_{m}(\mathbb{R}^{S}; \, \mathbb{R}^{S})$, mapping $\boldsymbol{\mathcal{P}}(\mathbf{x}; \, \mathbf{k})$ to $(\Psi_\mathcal{T} \boldsymbol{\mathcal{P}})(\mathbf{\bar{x}}; \, \mathbf{\bar{k}})$, is called a \emph{translation transformation}.
\end{definition}

\noindent 
A composition of a translation and a centroaffine transformation, 
$\Psi_{A,\mathcal{T}} = \Psi_{A} \circ \Psi_{\mathcal{T}}$, 
i.e. an \emph{affine transformation}, may be used as a possible kinetic transformation (see Definition~\ref{def:kinetictransf}). Let us note that condition (ii) in Definition~\ref{def:kinetictransf} is necessarily satisfied for all affine transformation, i.e. affine transformations preserve the topology of the phase space, as well as the polynomial degree of the functions being mapped~\cite{Escher}. For these reasons, affine transformations are preferred over the alternative nonlinear transformations, discussed in the next two sections. However, affine transformations do not necessarily satisfy condition (i) in Definition~\ref{def:kinetictransf}, so that they are generally nonkinetic transformations. However, despite being generally nonkinetic, affine transformations map sets $k$ into new sets $\bar{k}$ (see equations~(\ref{eqn:affine}) and~(\ref{eqn:translation})), so that they may be used for satisfying a given set of constraints imposed on the output function (see Definition~\ref{def:constraints}). This motivates the following definition.
\begin{definition} \label{def:essential}
Let $\mathcal{P}(x; \, k) \in \mathbb{P}_{m}(\mathbb{R}^{\mathcal{S}}; 
\, \mathbb{R}^{\mathcal{S}})$. If it is not possible that 
simultaneously
$(\Psi_{A} \circ \Psi_{\mathcal{T}} \mathcal{P})(\bar{x}; \, \bar{k})$ 
is a kinetic function, and that a given set of constraints 
$\Phi \equiv \{\phi_j(\bar{k}) \ge 0 : j = 1, 2, \ldots, J\}$ is satisfied, 
for all $A \in \mathbb{R}^{S \times S}$ and for all 
$\mathcal{T} \in \mathbb{R}^{S}$, then it is said that 
$\mathcal{P}(x; \, k)$ and the corresponding equation~(\ref{eqn:nk}) 
are \emph{affinely nonkinetic}, given the constraints. Otherwise, 
they are said to be \emph{affinely kinetic}, given the constraints. 
\end{definition} 
\noindent
If the set of constraints in Definition~\ref{def:constraints} is empty, 
affinely nonkinetic functions are called \emph{essentially nonkinetic}, 
while those that are affinely kinetic are called 
\emph{removably nonkinetic}. Such labels emphasize that, 
if a function is essentially nonkinetic, a kinetic function 
that is globally topologically equivalent cannot be obtained, 
while if a function is removably nonkinetic, a globally topologically 
equivalent kinetic function can be obtained.

Explicit sufficient conditions for a polynomial function 
$\boldsymbol{\mathcal{P}}(\mathbf{x}; \,\mathbf{k})$ to be affinely 
kinetic, or nonkinetic, are generally difficult to obtain. 
Even in the simpler case 
$\boldsymbol{\mathcal{P}}(\mathbf{x}; \,\mathbf{k}) 
\in \mathbb{P}_{2}(\mathbb{R}^{2}; \, \mathbb{R}^{2})$, 
such conditions are complicated, and cannot be easily generalized
for higher-dimensional systems and/or systems with higher 
degree of nonlinearily~\cite{Escher}. 
In~\cite{Toth4}, based on the polar and spectral 
decomposition theorems, it has been argued that if no 
orthogonal transformation is kinetic, then no centroaffine 
transformation is kinetic. The result is reproduced in this paper using 
the more concise singular value decomposition theorem, 
and is generalized to the case when the set of constraints is nonempty. 
Loosely speaking, the theorem states that ``orthogonally nonkinetic'' 
functions are affinely nonkinetic as well, given certain constraints.

\begin{theorem}\label{ThmAffine}
If $\mathcal{P}(x; \, k) \in \mathbb{P}_{m}(\mathbb{R}^{\mathcal{S}}; 
\, \mathbb{R}^{\mathcal{S}})$ is nonkinetic under 
$\Psi_{Q} \circ \Psi_{\mathcal{T}}$, given a set of constraints 
$\Phi$, for all orthogonal matrices $Q \in \mathbb{R}^{S \times S}$ 
and for all $\mathcal{T} \in \mathbb{R}^{S}$, then 
$\mathcal{P}(x; \, k)$ is also affinely nonkinetic, given $\Phi$, 
provided the following condition holds: 
$\mathrm{sign}(\phi_j(k)) = \mathrm{sign}(\phi_j(\bar{k}))$, 
$j = 1, 2, \ldots, J$, for all diagonal and positive definite 
matrices $\Lambda \in \mathbb{R}^{S \times S}$, 
with $\Psi_{\Lambda} \mathcal{P} = 
(\Psi_{\Lambda} \,\mathcal{P})(\bar{x}; \, \bar{k})$. 
\end{theorem}

\begin{proof}
By the singular value decomposition theorem, nonsingular matrices 
$A \in \mathbb{R}^{S \times S}$ can be written as $A = Q_1 \Lambda Q_2$, 
where $Q_1, Q_2 \in \mathbb{R}^{S \times S}$ are orthogonal, 
and  $\Lambda \in \mathbb{R}^{S \times S}$ diagonal and 
positive definite. Cross-negative terms are invariant under 
transformation $\Psi_{\Lambda}$ for all $\Lambda$~\cite{Toth4}. 
If $\Phi$ from Definition~\ref{def:constraints} is such that 
functions $\mathrm{sign}(\phi_j(k))$, $j = 1, 2, \ldots, J$, 
are invariant under all positive definite diagonal matrices 
$\Lambda \in \mathbb{R}^{S \times S}$, the statement of the 
theorem follows. 
\end{proof} 

\subsubsection{X-factorable transformation} 
\label{sec:xft}
\begin{definition} \label{def:xft}
Consider multipling the RHS of equation~(\ref{eqn:nkv}) by a 
diagonal matrix 
$\mathcal{X}(\mathbf{x}) = \mathrm{diag}(x_1, x_2, \ldots, x_S)$, 
resulting in
\begin{align}
\frac{\mathrm{d} \mathbf{x}}{\mathrm{d} t} & 
= \mathcal{X}(\mathbf{x}) 
\boldsymbol{\mathcal{P}}(\mathbf{x}; \, \mathbf{k})
\equiv
(\Psi_{\mathcal{X}} \boldsymbol{\mathcal{P}})(\mathbf{x}; \, \mathbf{k}).
\label{eqn:xft} 
\end{align}
Then 
$\Psi_{\mathcal{\mathcal{X}}}: \mathbb{P}_{m}(\mathbb{R}^{S}; 
\, \mathbb{R}^{S}) \to \mathbb{P}_{m + 1}(\mathbb{R}^{S}; \, \mathbb{R}^{S})$, 
mapping $\boldsymbol{\mathcal{P}}(\mathbf{x}; \, \mathbf{k})$ to 
$(\Psi_{\mathcal{X}} \boldsymbol{\mathcal{P}})(\mathbf{x}; \, \mathbf{k})$, 
is called an \emph{$x$-factorable transformation}. If $\mathcal{X}$ is diagonal
and its nonzero elements are
\begin{align}
\mathcal{X}_{s s} = \begin{cases}
x_s, & \textrm{if } s \in \mathcal{S'},\\
1, & \textrm{if } s \in \mathcal{S} \setminus \mathcal{S'}, \nonumber
\end{cases}
\end{align}
 where $\mathcal{S'} \subset \mathcal{S}$, 
$\mathcal{S'} \ne \emptyset$, then the transformation is denoted
$\Psi_{\mathcal{X}_{\mathcal{S'}}}$, and is said to be $x_{\mathcal{S'}}$-factorable.
\end{definition}

\noindent When $\mathcal{X} \in \mathbb{R}^2$ is $x_1$-factorable, i.e. 
$\mathcal{X}(x_1) = \mathrm{diag}(x_1,1)$, we write $\Psi_{\mathcal{X}_{1}} \equiv \Psi_{\mathcal{X}_{\{1\}}}$.

\begin{theorem} \label{thm:xfact}
$(\Psi_{\mathcal{X}} \boldsymbol{\mathcal{P}})(\mathbf{x}; \, \mathbf{k})$ 
from \emph{Defnition~\ref{def:xft}} is a \emph{kinetic function}, i.e. 
$(\Psi_{\mathcal{X}} \boldsymbol{\mathcal{P}})(\mathbf{x}; 
\, \mathbf{k}) \in \mathbb{P}^{\mathcal{K}}_{m + 1}(\mathbb{R}_{\ge}^{S}; \, 
\mathbb{R}^{S})$.
\end{theorem}
\begin{proof}
See~\cite{Xfactor}.
\end{proof}

\noindent
Functions $\boldsymbol{\mathcal{P}}(\mathbf{x}; \, \mathbf{k})$ 
and $(\Psi_{\mathcal{X}} \boldsymbol{\mathcal{P}})(\mathbf{x}; \, \mathbf{k})$ 
are not necessarily topologically equivalent due to two overlapping 
artefacts that $\Psi_{\mathcal{X}}$ can produce, so that $\Psi_{\mathcal{X}}$ 
is generally a nonkinetic transformation. Firstly, the fixed points 
of the former system can change the type and/or stability under
$\Psi_{\mathcal{X}}$, and, secondly, the latter system has an 
additional finite number of boundary fixed points. The following 
theorem specifies the details of the artefacts for two-dimensional 
systems.

\begin{theorem}\label{Xfactdet}
Let us consider the ODE system~$(\ref{eqn:nkv})$ in two 
dimensions with RHS
$
\boldsymbol{\mathcal{P}}(\mathbf{x}; \, \mathbf{k})
= 
(\mathcal{P}_1(\mathbf{x}; \, \mathbf{k}),
 \mathcal{P}_2(\mathbf{x}; \, \mathbf{k}))^{\top}.
$
The following statements are true for all the fixed 
points $\mathbf{x}^{*}$ of the two-dimensional system~$(\ref{eqn:nkv})$ 
in $\mathbb{R}_{>}^2$ under $\Psi_{\mathcal{X}_{\mathcal{S'}}}$, 
$\mathcal{S'} \subseteq \mathcal{S}$, $\mathcal{S'} \ne \emptyset$:
\begin{enumerate}
\item[\emph{(i)}] 
All the saddle fixed points are unconditionally invariant,
i.e. saddle points of~$(\ref{eqn:nkv})$ correspond to saddle
points of~$(\ref{eqn:xft})$.
\item[\emph{(ii)}]  
A sufficient condition for stability of a fixed point 
$\mathbf{x}^{*}$ to be invariant is: 
$$
\frac{
\partial \mathcal{P}_1(\mathbf{x}; \, \mathbf{k})
}{\partial x_1}|_{\mathbf{x}
= \mathbf{x}^{*}} 
\frac{
\partial \mathcal{P}_2(\mathbf{x}; \, \mathbf{k})
}{\partial x_2}|_{\mathbf{x} = \mathbf{x}^{*}} \ge 0
.$$
\item[\emph{(iii)}] 
A sufficient condition for the type of a fixed point 
$\mathbf{x}^{*}$ to be invariant is: 
$$
\frac{
\partial \mathcal{P}_1(\mathbf{x}; \, \mathbf{k})
}{\partial x_2}|_{\mathbf{x} = \mathbf{x}^{*}} 
\frac{
\partial \mathcal{P}_2(\mathbf{x}; \, \mathbf{k})
}{\partial x_1}|_{\mathbf{x} = \mathbf{x}^{*}} 
\ge 0.
$$
\end{enumerate}
Assume that the ODE system~$(\ref{eqn:nkv})$ 
does not have fixed points on the axes of the phase space.
Nevertheless, the two-dimensional system~$(\ref{eqn:xft})$ can have 
additional fixed points on the axes of the phase space, called 
\emph{boundary fixed points}, denoted 
$\mathbf{x}_{b}^{*} \in \mathbb{R}_{\ge}^{2}$.
The boundary fixed points can be either nodes or saddles, 
and the following statements are true:
\begin{enumerate}
\item[\emph{(iv)}] 
If system~$(\ref{eqn:xft})$ is $x$-factorable, then the 
origin is a fixed point, $\mathbf{x}_{b}^{*} = \mathbf{0}$, with 
eigenvalues 
$\lambda_i = \mathcal{P}_i(\mathbf{x}_{b}^{*}; \, \mathbf{k}) \ne 0$, 
$i = 1,2$, and the corresponding eigenvectors along the phase 
space axes. 
\item[\emph{(v)}] 
For $x_{b, i}^* = 0$, $x_{b, j}^* \ne 0$, 
$\mathbf{x}_{b}^{*} \in \mathbb{R}_{\ge}^2$, 
$i,j =1,2,$ $i \ne j$, a boundary fixed 
point is a node if and only 
if 
$$
\mathcal{P}_i(\mathbf{x}_{b}^{*}; \, \mathbf{k}) 
\frac{
\partial \mathcal{P}_{j}(\mathbf{x}; \, \mathbf{k})
}{\partial x_{j}}|_{\mathbf{x} = \mathbf{x}_{b}^{*}} > 0,
$$ 
with the node being stable if 
$\mathcal{P}_i(\mathbf{x}_{b}^{*}; \, \mathbf{k}) < 0$, 
and unstable if $\mathcal{P}_i(\mathbf{x}_{b}^{*}; \, \mathbf{k}) > 0$, 
$i,j = 1, 2$, $i \ne j$. Otherwise, the fixed point is a saddle.  
\end{enumerate}
\end{theorem}

\begin{proof}
Without loss of generality, we consider two forms of
the system~(\ref{eqn:xft}) with $S = 2$:
\begin{align}
\frac{d x_1}{d t} & 
= x_1 \mathcal{P}_1(\mathbf{x}; \, \mathbf{k}), 
\label{Xfact2Da}
\\ 
\frac{d x_2}{d t} & 
= x_2^p \mathcal{P}_2(\mathbf{x}; \, \mathbf{k})
\label{Xfact2Db},
\end{align}
where $p \in \{0,1\}$, so that system~(\ref{Xfact2Da})--(\ref{Xfact2Db}) 
is $x$-factorable for $p = 1$, but only $x_1$-factorable for $p = 0$. 
The results derived for an $x_1$-factorable system hold when 
the system is $x_2$-factorable, if the indices are swapped. 
By writing $\mathcal{P}_i(\mathbf{x}; \, \mathbf{k}) = \mathcal{P}_i$, 
$i = 1,2$, the Jacobian of~(\ref{Xfact2Da})--(\ref{Xfact2Db}),
$J_{\mathcal{X}}$, is for $p \in \{0,1\}$ given by 
$$
J_{\mathcal{X}}(\mathbf{x}) =  \left(\begin{array}{cc}
\mathcal{P}_1 + x_1 \frac{\partial \mathcal{P}_1}{\partial x_1}  
& x_1 \frac{\partial \mathcal{P}_1}{\partial x_2} \label{JacobX1} 
\\
\rule{0pt}{4mm}
x_2^p \frac{\partial \mathcal{P}_2}{\partial x_1}  
& p \mathcal{P}_2 + x_2^p \frac{\partial \mathcal{P}_2}{\partial x_2} 
\end{array}
\right).
$$
First, consider how fixed points of
$\boldsymbol{\mathcal{P}}(\mathbf{x}; \, \mathbf{k})$ 
are affected by transformation $\Psi_{\mathcal{X}_{\mathcal{S'}}}$.
Denoting the Jacobian of two-dimensional system~(\ref{eqn:nkv}) by $J$, 
and assuming the fixed points are not on the axes of the 
phase space (i.e. $\mathbf{x}^* \in \mathbb{R}_{>}^2$), the Jacobians 
evaluated at $\mathbf{x}^*$ are given by:
$$
J(\mathbf{x}^*) 
=  \left(\begin{array}{cc}
\frac{\partial \mathcal{P}_1}{\partial x_1}  
&  
\frac{\partial \mathcal{P}_1}{\partial x_2}
\\
\rule{0pt}{4mm}
\frac{\partial \mathcal{P}_2}{\partial x_1}  
&  \frac{\partial \mathcal{P}_2}{\partial x_2} 
\end{array}
\right)\Big|_{\mathbf{x} = \mathbf{x}^*}, \qquad
J_{\mathcal{X}}(\mathbf{x}^*)  =  \left(\begin{array}{cc}
x_1 \frac{\partial \mathcal{P}_1}{\partial x_1}  
& x_1 \frac{\partial \mathcal{P}_1}{\partial x_2} \label{JacobX2} 
\\
\rule{0pt}{4mm}
x_2^p \frac{\partial \mathcal{P}_2}{\partial x_1}  
& x_2^p \frac{\partial \mathcal{P}_2}{\partial x_2} 
\end{array}
\right)\Big|_{\mathbf{x} = \mathbf{x}^*}.
$$
Comparing the trace, determinant and discriminant of $J(\mathbf{x}^*)$ 
and $J_{\mathcal{X}}(\mathbf{x}^*)$, we deduce (i)--(iii).

To prove (iv)--(v), we evaluate $J_{\mathcal{X}}$ at the boundary 
fixed points of the form $\mathbf{x}_{b}^* = (0,x_{b, 2}^*)$ to get
\begin{align}
J_{\mathcal{X}}(\mathbf{x}_{b}^*) 
& =  \left(\begin{array}{cc}
\mathcal{P}_1 & 0  
\\
x_2^p \frac{\partial \mathcal{P}_2}{\partial x_1} 
& \; \; p \mathcal{P}_2 + x_2^p \frac{\partial \mathcal{P}_2}{\partial x_2} 
\end{array}
\right)\Big|_{\mathbf{x} = \mathbf{x}_{b}^*}.
\label{boundjac}
\end{align}
If $p = 1$, then one of the boundary fixed points is 
$\mathbf{x}_{b}^* = \mathbf{0}$, and the Jacobian becomes 
a diagonal matrix, so that condition (iv) holds.
If $x_{b, 2}^* \ne 0$, then $\mathcal{P}_2(0,x_{b, 2}^*;\mathbf{k})=0$
in (\ref{boundjac}), and comparing the trace, determinant and discriminant of $J(\mathbf{x}^*)$ 
and $J_{\mathcal{X}}(\mathbf{x}_b^*)$, we deduce (v).
\end{proof}

\noindent
Theorem~\ref{Xfactdet} can be used to find conditions that an $x$-factorable 
transformation given by $\Psi_{\mathcal{\mathcal{X}}}: \mathbb{P}_{m}(\mathbb{R}^{2}; 
\, \mathbb{R}^{2}) \to \mathbb{P}_{m + 1}(\mathbb{R}^{2}; \, \mathbb{R}^{2})$ 
is a kinetic transformation. While conditions (ii)--(iii) in Theorem~\ref{Xfactdet} 
may be violated when $\Psi_{\mathcal{X}}$ is used, so that $\Psi_{\mathcal{X}}$ 
is a nonkinetic transformation, a composition of an affine transformation and 
an $x$-factorable transformation, i.e. 
$\Psi_{\mathcal{X}, A,\mathcal{T}} =
\Psi_{\mathcal{X}} \circ \Psi_{A} \circ \Psi_{\mathcal{T}}$, 
may be kinetic. Furthermore, such a composite transformation may also be used 
to control the boundary fixed points introduced by $\Psi_{\mathcal{X}}$. 
Finding an appropriate $A$ and $\mathcal{T}$ to ensure the topological 
equivalence near the fixed points typically means that the region of 
interest in the phase space has to be positioned at a sufficient 
distance from the axes. However, since the introduced boundary fixed 
points may be saddles, this implies that the phase curves may 
be significantly globally changed, regardless of how far away 
from the axes they are. The most desirable outcome of controlling 
the boundary fixed points is to eliminate them, or shift them 
outside of the nonnegative orthant. The former can be attempted 
by ensuring that the nullclines of the original ODE system~(\ref{eqn:nkv}) 
do not intersect the axes of the phase space, while the 
latter by using the Routh-Hurwitz theorem~\cite{Matrices}.

An alternative transformation, which is always kinetic, that also 
does not change the dimension of an ODE system is 
the time-parametrization transformation~\cite{QSSA4}.
However, while $\Psi_{\mathcal{X}}$ increases the polynomial degree 
by one, and introduces only a finite number of boundary fixed points 
which are given as solutions of suitable polynomials, the 
time-parameterization transformation generally increases the 
nonlinearity degree more than $\Psi_{\mathcal{X}}$,
and introduces infinitely many boundary fixed points.

\subsubsection{The quasi-steady state transformation} 
\label{sec:QSSA}

The quasi-steady state assumption (QSSA) is 
a popular \emph{constructive} method for reducing dimension 
of ODE systems by assuming that, at a given time-scale, some 
of the species reach a quasi-steady state, so that they can 
be described by algebraic, rather than differential equations. 
The QSSA is based on Tikhonov's theorem~\cite{QSSA1,QSSA2} 
that specifies conditions ensuring that the solutions of the 
reduced system are asymptotically equivalent to the solutions 
of the original system. The original system is referred to as 
the \emph{total system}, and it consists of the reduced subsystem, 
referred to as the \emph{degenerate system}, and the remaining 
subsystem, called the \emph{adjointed system}, so that the QSSA 
consists of replacing the total system with the degenerate one, 
by eliminating the adjointed system. Korzukhin's 
theorem~\cite{QSSA1,QSSA2} is an \emph{existence} result 
ensuring that, given any polynomial degenerate system, there exists 
a corresponding total system that is kinetic. 

Thus, Tikhonov's theorem can be seen as a constructive direct 
asymptotic dimension reduction procedure, while Korzukhin's theorem 
as an inverse asymptotic dimension expansion existence result.
Korzukhin's theorem has an important implication that an
application of the QSSA can result in a degenerate system that 
is \emph{structurally} different than the corresponding total 
system. In this paper, the QSSA 
is assumed to necessarily be compatible with Tikhonov's theorem.
If this is not the case, then it has been demonstrated
in~\cite{QSSA4,QSSA5} that application of a QSSA 
can create dynamical artefacts, i.e. it can result in degenerate systems, 
not only structurally different, but also \emph{dynamically} different 
from the total systems. The artefacts commonly occur due to  
the asymptotic parameters in Tikhonov's theorem not being sufficiently 
small. For example, it has been shown that exotic phenomena such as
multistability and oscillations can exist in a degenerate system, 
while not existing in the corresponding total system~\cite{QSSA4,QSSA5}. 

Using Korzukhin's and Tikhonov's theorems, a family of kinetic total 
systems for an arbitrary nonkinetic polynomial degenerate system 
can be constructed, as is now shown. For simplicity,
we denote
$
x^c = \prod_{s \in S} x_s^{c_s},
$
for any $x \in \mathbb{R}^{\mathcal{S}}$ and 
$c \in \mathbb{N}^{\mathcal{S}}.$

\begin{definition} \label{def:QSSA}
Consider equation~(\ref{eqn:nk}), and assume that the reaction set 
is partitioned, $\mathcal{R} = \mathcal{R}_1 \cup \mathcal{R}_2$,
$\mathcal{R}_1 \cap \mathcal{R}_2 = \emptyset$, so that~(\ref{eqn:nk}),
together with the initial conditions, can be written as
\begin{align}
\frac{\mathrm{d} x_s}{\mathrm{d} t} 
& =  \sum_{r \in \mathcal{R}_1} a_{s r} x^{\alpha_{s r}} 
- 
\sum_{r \in \mathcal{R}_2} b_{s r} x^{\beta_{s r}}, 
\qquad \mbox{for} \; s \in \mathcal{S},
\label{gpolya} \\
x_s(t_0) & = x_s^0, \qquad x_s^0 \ge 0,
\label{gpolyb} 
\end{align}
where $x \in \mathbb{R}^{\mathcal{S}}$, 
$\alpha_{s r} \in \mathbb{N}^{\mathcal{S}},$ 
$\beta_{s r} \in \mathbb{N}^{\mathcal{S}},$ 
$\alpha_{s r} \ne \beta_{s r}$, 
$a_{s r} \in \mathbb{R}_{\ge}$
and 
$b_{s r} \in \mathbb{R}_{\ge}$ 
for all $s \in \mathcal{S}$ 
and $r \in \mathcal{R}$.
Assume further that the species set is partitioned, 
$\mathcal{S} = \mathcal{S}_1 \cup \mathcal{S}_2$, 
$\mathcal{S}_1 \cap \mathcal{S}_2 = \emptyset$, so that equations 
for species $s \in \mathcal{S}_1$ are kinetic, while those for 
species $s \in \mathcal{S}_2$ are nonkinetic. Consider the 
following total system, consisting of a degenerate system given by
\begin{align}
\frac{\mathrm{d}  x_s}{\mathrm{d}  t} 
& =  
\sum_{r \in \mathcal{R}_1} a_{s r} x^{\alpha_{s r}} 
- \sum_{r \in \mathcal{R}_2} b_{s r} x^{\beta_{s r}}, 
& \mbox{for} \; s \in \mathcal{S}_1, 
\label{gkin21a} \\
\frac{\mathrm{d}  x_s}{\mathrm{d}  t} & 
=   \sum_{r \in \mathcal{R}_1} a_{s r} x^{\alpha_{s r}} 
- \omega_s^{-1} x_s p_s(x) y_s \sum_{r \in \mathcal{R}_2} b_{s r} 
x^{\beta_{s r}},
& \mbox{for} \; s \in \mathcal{S}_2, 
\label{gkin21b} 
\end{align}
which satisfies the initial condition (\ref{gpolyb}) with
$x_s^0 > 0$ for $s \in \mathcal{S}_2$, and an
adjointed system given by
\begin{align}
\mu \frac{\mathrm{d}  y_s}{\mathrm{d}  t} & 
= \omega_s - x_s p_s(x) y_s , 
& \mbox{for} \; s \in \mathcal{S}_2, 
\label{gkin22a}
\\
y_s(t_0) &= y_s^0, \qquad y_s^0 \ge 0,  
& \mbox{for} \; s \in \mathcal{S}_2, 
\label{gkin22b} 
\end{align}
where $\mu > 0$, $\omega_s > 0$ are parameters, 
and $p(x)$ is a polynomial function of $x$ satisfying
$p(x) \in \mathbb{P}_{m_0}(\mathbb{R}_{\ge}^{\mathcal{S}}; 
\, \mathbb{R}_{>}^{\mathcal{S}_2})$ for $m_0 \in \mathbb{N}.$ 
Then 
$\Psi_{\mathrm{QSSA}} : \mathbb{P}_{m}(\mathbb{R}^{\mathcal{S}}; 
\, \mathbb{R}^{\mathcal{S}}) \to 
\mathbb{P}_{\bar{m}}(\mathbb{R}_{\ge}^{\mathcal{S} + \mathcal{S}_2}; 
\, \mathbb{R}^{\mathcal{S} + \mathcal{S}_2})$, mapping the 
RHS of differential equations in system~(\ref{gpolya})--(\ref{gpolyb}), 
denoted $\mathcal{P}(x; \, k)$, to the RHS of differential equations 
of system~(\ref{gkin21a})--(\ref{gkin22b}), denoted 
$(\Psi_{\mathrm{QSSA}} \mathcal{P})(\{x, y\}; \, \bar{k})$, 
with the constraint that $x_s > 0$ for $s \in \mathcal{S}_2$, is 
called a \emph{quasi-steady state transformation}. Here, 
$\bar{m} \le m + m_0 + 2$, and $\bar{k}$ is a vector of the 
new rate constants obtained from $k$
by rewriting the polynomial 
$(\Psi_{\mathrm{QSSA}} \mathcal{P})(\{x, y\}; \, \bar{k})$
into the mass action form.
\end{definition}

\begin{theorem}\label{TheoremADD}
Solutions of systems~$(\ref{gpolya})$--$(\ref{gpolyb})$
and~$(\ref{gkin21a})$--$(\ref{gkin22b})$, corresponding
to $\mathcal{P}(x;k)$ and 
$(\Psi_{\mathrm{QSSA}} \mathcal{P})(\{x, y\}; \, \bar{k})$,
are \emph{asymptotically equivalent} in the limit $\mu \rightarrow 0$, 
and $(\Psi_{\mathrm{QSSA}} \mathcal{P})(\{x, y\}; \, \bar{k})$ is a \emph{kinetic function}.
\end{theorem}

\begin{proof}
Fixed points of system~(\ref{gkin22a}) 
satisfy $y_s^{*} = \omega_s (x_s p_s(x) )^{-1}$. The fixed points are isolated, and, since (from Definition~\ref{def:QSSA}) $x_s > 0$ and 
$p_s(x) > 0$, $\forall x \in \mathbb{R}_{\ge}^{\mathcal{S}}$, 
$\forall s \in \mathcal{S}_2$, it follows that the fixed points are globally stable. Thus, the conditions of Tikhonov's theorem~\cite{QSSA1} are satisfied by the total
system~(\ref{gkin21a})--(\ref{gkin22b}). Then, by applying the theorem, i.e. substituting $y_s^{*}$ into~(\ref{gkin21b}), one recovers the corresponding degenerate system given 
by~(\ref{gpolya})--(\ref{gpolyb}). Finally, the total system~(\ref{gkin21a})--(\ref{gkin22b}) is kinetic, as can be verified by using Definition~\ref{def:cnt}.
\end{proof}

\begin{corollary} \label{cor:qssa}
The quasi-steady state transformation $\Psi_{\mathrm{QSSA}}$, defined in \emph{Definition~\ref{def:QSSA}}, is a \emph{kinetic transformation} 
in the limit $\mu \rightarrow 0$.
\end{corollary}

\noindent
An alternative transformation, for which condition (i) in 
Definition~\ref{def:kinetictransf} is also always satisfied, and that 
also expands the dimension of an ODE system,
is an incomplete Carleman embedding~\cite{UNI2,CARL1}. However, 
condition (ii) in Definition~\ref{def:kinetictransf} is satisfied 
for the incomplete Carleman embedding only provided initial conditions 
for the adjointed system are appropriately constrained, and, furthermore, 
the transformation generally results in an adjointed system with a 
higher nonlinearity degree when compared to $\Psi_{\mathrm{QSSA}}$. 
In fact, Theorem~\ref{TheoremADD} can be seen as an asymptotic 
alternative to the incomplete Carleman embedding, i.e. instead 
of requiring adjointed variables to satisfy 
$y_{s}(t) = \omega_s \, x_s^{-1}(t) \, p_s^{-1}(x(t))$, 
one requires them to satify
$\lim_{\mu \to 0} y_{s}(t) = \omega_s \, x_s^{-1}(t) \, p_s^{-1}(x(t))$. 
The theorem can also be seen as an extension of using the 
QSSA to represent reactions of more than two molecules as a limiting case 
of bimolecular reactions~\cite{UNI3} to the case of using the QSSA 
to represent cross-negative terms as a limiting case of kinetic ones.

\section{Construction of reaction systems undergoing a 
supercritical homoclinic bifurcation}
\label{sec:construction}

In this section, a brief review of a general bifurcation theory, 
and a more specific homoclinic bifurcation, is presented. This is 
followed by applying the framework developed in Section~\ref{sec:inverse} 
to construct specific reaction systems displaying the homoclinic 
bifurcation. 

Variations of parameters in a parameter dependent ODE system 
may change topology of the phase space, e.g. a change may occur 
in the number of invariant sets or their stability, shape of 
their region of attraction or their relative position. At 
values of the bifurcation parameters at which the system becomes 
topologically nonequivalent it is said that a bifurcation occurs, 
and the bifurcation is characterized by two sets of conditions: 
bifurcation conditions defining the type of bifurcation, and 
genericity conditions ensuring that the system is generic, 
i.e. can be simplified near the bifurcation to a normal
form~\cite{Bifur1}. If it is sufficient to analyse 
a small neighbourhood of an invariant set to detect a 
bifurcation, the bifurcation is said to be local. Otherwise, 
it is called global, and the analysis becomes more challenging. 
Bifurcations are common in kinetic equations, where, in the 
case of the mass action kinetics, the rate constants play 
the role of bifurcation
parameters~\cite{Intro1,Intro11,Intro2,Intro3,Intro4}. 
In this paper, focus is placed on a global bifurcation 
giving rise to stable oscillations, called a supercritical 
homoclinic bifurcation~\cite{Bifur1,Sand,Intro11}.

\begin{definition}
Suppose $\mathbf{x}^{*}$ is a fixed point of system~(\ref{eqn:nkv}). 
An orbit $\gamma^{*}$ starting at a point $\mathbf{x}$ is 
called \emph{homoclinic} to the fixed point $\mathbf{x}^{*}$ if 
its $\alpha$-limiting and $\omega$-limiting sets are both 
equal to $\mathbf{x}^{*}$. 
\end{definition}

\noindent
Put more simply, a homoclinic orbit connects a fixed point 
to itself. An example of a homoclinic orbit to a saddle 
fixed point can be seen in Figure~\ref{fig:alpha}(b) on 
page~\pageref{fig:alpha}, where 
the homoclinic orbit is shown as the purple loop, while the 
saddle as the blue dot at the origin.

If a homoclinic orbit to a hyperbolic fixed point is present 
in an ODE system, then the system is \emph{structurally unstable}, 
i.e. small perturbations to the equations can destroy the 
homoclinic orbit and change the structure of the phase space, 
so that a bifurcation can occur. For two-dimensional systems, the bifurcation 
and genericity conditions are completely specified by the 
Andronov-Leontovich theorem~\cite{Bifur1} given
in~\ref{HomoclinicTheorem}. In summary, the theorem 
demands that the sum of the eigenvalues corresponding to the 
saddle at the bifurcation point, called the saddle quantity, 
must be nonzero (nondegeneracy condition), and that the 
so-called Melnikov integral at the bifurcation point evaluated 
along the homoclinic orbit must be nonzero (transversality condition).

\subsection{The inverse problem formulation} 
\label{sec:formulation}
Construction of reaction systems with prescribed properties 
is an inverse problem which we will solve by applying kinetic 
transformations described in Section~\ref{sec:inverse}. 
Our goal is to find a reaction system with the mass action kinetics 
(see Definition~\ref{def:mak}) such that the kinetic equations 
satisfy assumptions of Andronov-Leontovich theorem in~\ref{HomoclinicTheorem}, 
i.e. they must contain a homoclinic orbit defined on a two-dimensional 
manifold in the nonnegative orthant, and undergo the homoclinic 
bifurcation in a generic way. The output of this inverse problem
will be a canonical reaction network which corresponds to the 
constructed ODE system. Thus the inverse problem is solved in 
three steps given in Algorithm~\ref{tabinv}. The first step is solved
by using results of Sandstede~\cite{Sand} which leads to 
a set of polynomial functions satisfying the first three conditions 
of Andronov-Leontovich theorem in~\ref{HomoclinicTheorem}.
An additional transformation is then performed ensuring that the 
final condition of Andronov-Leontovich theorem is satisfied. 
In this paper, nonlinear kinetic transformations are applied 
on the resulting polynomial function (using Step (2), case (b), 
in Algorithm~\ref{tabinv}).

\begin{table}[t]
 \renewcommand\tablename{Algorithm}
\hrule
\vskip 2.5 mm
\begin{enumerate}
\item \textbf{Construction of a polynomial function} $\mathcal{P}(x; \, k)$: \\
Find an ODE system (\ref{eqn:nk}) which satisfies the assumptions 
of Andronov-Leontovich theorem in~\ref{HomoclinicTheorem}. 
\item \textbf{Construction of a kinetic function} 
$\mathcal{K}(\bar{x}; \, \bar{k})$: \\
Find a transformation so that the following conditions are satisfied:
\begin{enumerate}[(i)]
\item The transformation is kinetic (see Definition~\ref{def:kinetictransf}), 
mapping the polynomial function $\mathcal{P}(x; \, k)$ into a kinetic 
function 
$\mathcal{K}(\bar{x}; \, \bar{k}) \equiv (\Psi \mathcal{P})(\bar{x}; \, \bar{k})$. 
\item The set of constraints (see Definition~\ref{def:constraints}) 
ensuring that the homoclinic orbit is in $\mathbb{R}_{\ge}^2$ 
are satisfied for $\mathcal{K}(\bar{x}; \, \bar{k})$. 
\end{enumerate}
To determine the choice of 
$\Psi$, if possible, use Theorem~\ref{ThmAffine} to deduce that 
$\mathcal{P}(x; \, k)$ is affinely nonkinetic 
(see Definition~\ref{def:essential}), given the constraints. 
\begin{enumerate}
\item If $\mathcal{P}(x; \, k)$ is not affinely nonkinetic, attempt 
to find an affine transformation
$\Psi = \Psi_{A}$ such that (i)--(ii) are satisfied.
\item If $\mathcal{P}(x; \, k)$ is affinely nonkinetic, or 
if application of Theorem~\ref{ThmAffine} is computationally 
too complicated, then choose kinetic transformation
$\Psi$ satisfying (i)--(ii) as an appropriate composition of $\Psi_{A}$, 
$\Psi_{\mathcal{X}}$ and $\Psi_{\mathrm{QSSA}}$, where 
$\Psi_{\mathcal{X}}$ is an $x$-factorable transformation 
(see Section~\ref{sec:xft}) and $\Psi_{\mathrm{QSSA}}$ is 
a quasi-steady state transformation 
(see Section~\ref{sec:QSSA}, in particular Corollary~\ref{cor:qssa}).
\end{enumerate}
\item \textbf{Construction of a reaction network}: \\
Use Definition~\ref{def:cn} to construct the canonical reaction network $\mathcal{R}_{\mathcal{K}^{-1}}$.
\end{enumerate}
\hrule
\caption{{\it Three steps of the solution to the inverse problem of
finding reaction systems undergoing a supercritical homoclinic 
bifurcation.}}
\label{tabinv}
\end{table}

\subsection{Step \emph{(1)}: construction of polynomial function
$\mathcal{P}(x; \, k)$}  
\label{sec:step1}

\begin{definition} \label{alphadefinition}
A version of a plane algebraic curve called Tschirnhausen cubic 
(also known as Catalan's trisectrix, and l'Hospital's 
cubic)~\cite{Bifur2} given by:
\begin{align}
H(x_1,x_2) & =  -x_1^2 + x_2^2 (1 + x_2) \; \; = \; \; 0, \label{loop}
\end{align}
is referred to as the \emph{alpha curve}. The part of the 
curve with $x_2 < 0$ is called the \emph{alpha loop}, while 
the part with $x_2 > 0$ is called the \emph{alpha branches}, 
with the branch for $x_1 < 0$ being the \emph{negative alpha branch}, 
and for $x_1 > 0$ being the \emph{positive alpha branch}. 
Solutions $x_1$ of equation~(\ref{loop}) are denoted 
$x_1^{\pm} = \pm x_2 \sqrt{1 + x_2}$.
\end{definition}

\begin{figure}[t]
\centerline{
\hskip -2mm
\includegraphics[width=0.46\columnwidth]{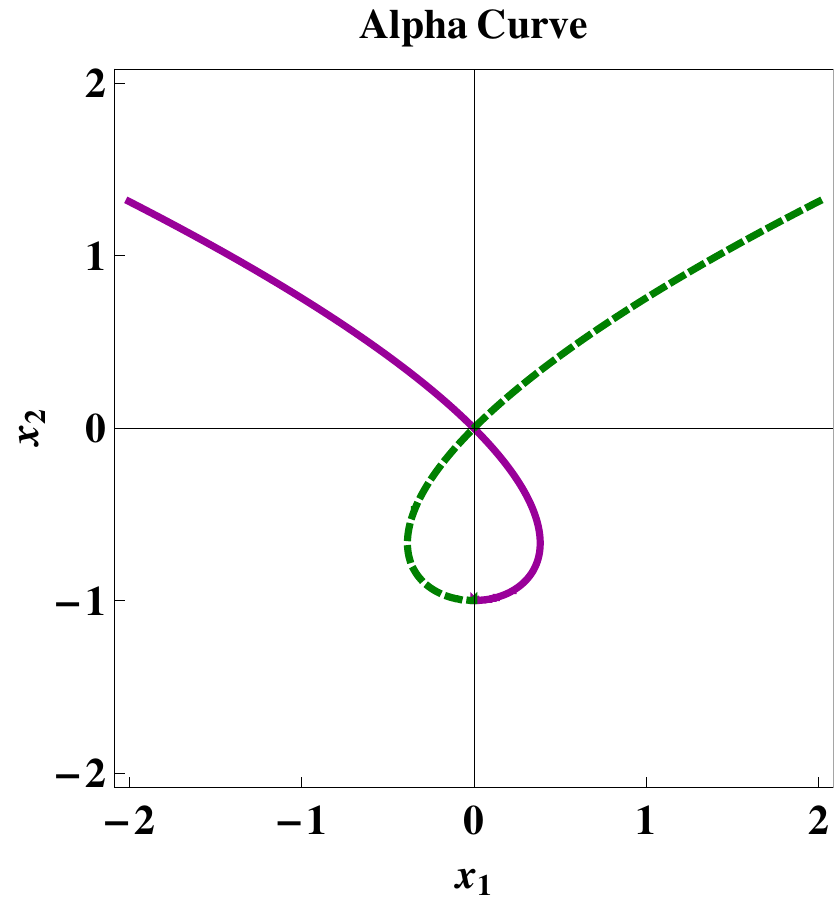}
\hskip 6mm
\includegraphics[width=0.46\columnwidth]{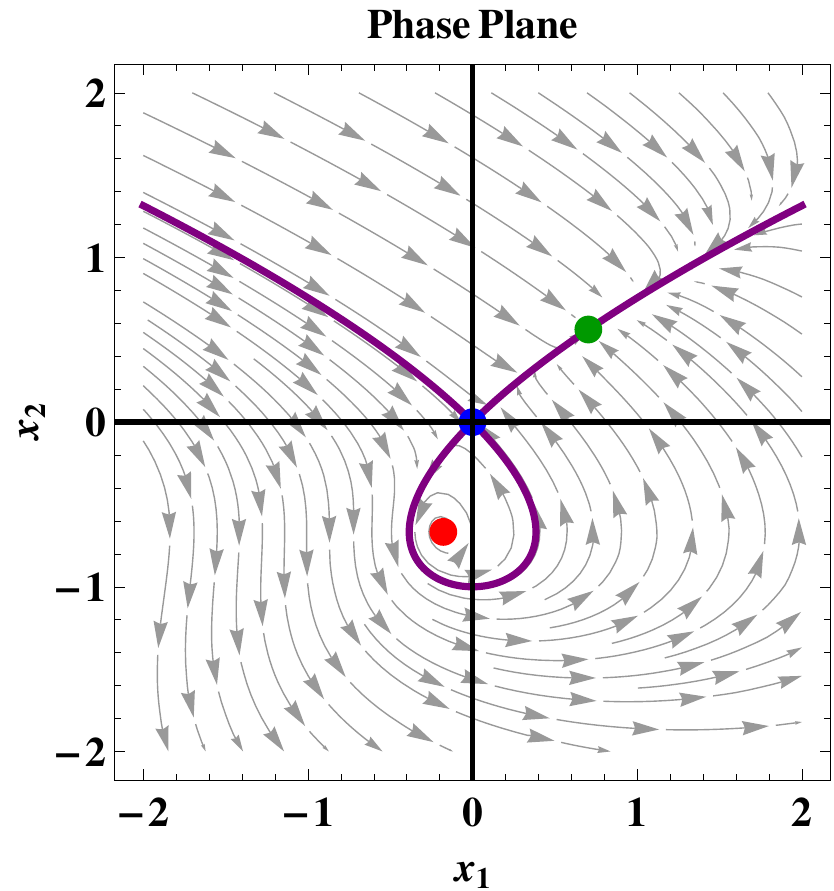}
}
\vskip -6.5cm
\leftline{\hskip 1mm (a) \hskip 5.95cm (b)}
\vskip 6.0cm
\caption{{\rm (a)} 
{\it The alpha curve~$(\ref{loop})$, with branch $x_1^{-}$ plotted as 
the solid purple curve, and branch $x_1^{+}$ as the dashed green curve.} 
{\rm (b)} 
{\it Phase plane diagram of system~$(\ref{ansatza})$--$(\ref{ansatzb})$ 
for $a = -0.8$, with the 
stable node, the saddle, and the unstable spiral represented as 
the green, blue and red dots, respectively. The alpha curve is 
shown in purple, while the vector field as gray arrows.}}
\label{fig:alpha}
\end{figure}

\noindent
The $\alpha$ curve is shown in Figure~\ref{fig:alpha}(a), 
with $x_1^{-}$ plotted as the solid purple curve, and $x_1^{+}$ 
as the dashed green curve. It can be seen that the curve 
consists of the tear-shaped alpha loop located in region 
$\lbrack -2 \sqrt{3}/9,2 \sqrt{3}/9 \rbrack \times \lbrack -1,0 \rbrack$, 
and the positive and negative alpha branch in the first and 
second quadrant, respectively. As is done in~\cite{Sand}, 
the alpha curve is mapped into a system of polynomial ODEs.

\begin{lemma}\label{Theorem1}
The two-dimensional polynomial \emph{ODE} system
\begin{align}
\frac{\mathrm{d} x_1}{\mathrm{d}  t} \; \; 
= \mathcal{P}_1(\mathbf{x}; \, a) & 
= a x_1 +  x_2 + \frac{3}{2} a x_1 x_2 + \frac{3}{2} x_2^2 , 
\label{ansatza} \\
\frac{\mathrm{d}  x_2}{\mathrm{d}  t}  \; \; 
= \mathcal{P}_2(\mathbf{x}; \, a) & =  x_1 + a x_2 + a x_2^2,
\label{ansatzb}
\end{align}
contains the alpha curve~$(\ref{loop})$ as phase plane orbits, 
with the alpha loop as a homoclinic orbit to the fixed 
point $\mathbf{x}^{*} = \mathbf{0}$, provided $a^2 < 1$. 
If $a \in (-1, 0)$, the alpha loop is stable from the inside, 
and system~$(\ref{ansatza})$--$(\ref{ansatzb})$ has three fixed 
points: a saddle at the origin, an unstable spiral inside 
the alpha loop, and a stable node on the positive alpha branch.
\end{lemma}

\begin{proof}
Setting 
$\mathbf{\boldsymbol{\mathcal{P}}}(\mathbf{x}; \, \mathbf{k}) 
= (\mathcal{P}_1(\mathbf{x}; \, \mathbf{k}),
   \mathcal{P}_2(\mathbf{x}; \, \mathbf{k}))$ in 
system~(\ref{eqn:nkv}) to be a polynomial function of 
$\mathbf{x} =(x_1,x_2)$ with undetermined coefficients $\mathbf{k}$, 
and requiring $\boldsymbol{\mathcal{P}} \cdot \nabla H = 0$, 
one obtains system~(\ref{ansatza})--(\ref{ansatzb}), as was 
done in~\cite{Sand}. As there is only one free parameter, 
denoted $a$, we write:
$\mathbf{\boldsymbol{\mathcal{P}}}(\mathbf{x}; \, \mathbf{k}) 
= \mathbf{\boldsymbol{\mathcal{P}}}(\mathbf{x}; \, a)$.
System~(\ref{ansatza})--(\ref{ansatzb}) has three fixed points: 
$\big(0, 0 \big)$, 
$\big(2a/9, -2/3\big)$, 
and 
$\big(a^{-1} (1 - a^{-2}), \, -1 + a^{-2} \big)$.
The condition $a^2 < 1$ ensures that fixed points
$\big(2a/9, -2/3\big)$ and 
$\big(a^{-1} (1 - a^{-2}), \, -1 + a^{-2} \big)$
are not on the alpha loop. The Jacobian 
$J = \nabla \mathbf{\mathcal{P}}(\mathbf{x}; \, a)$ 
is given by
\begin{align}
J &=  \left(\begin{array}{cc}
a + \frac{3}{2} a x_2 & \; 1 + \frac{3}{2} a x_1 + 3  x_2
\\
1 & a + 2 a x_2
\end{array}
\right). \label{Jacobian1}
\end{align}
Let the determinant and trace of $J$ be denoted 
by $\mathrm{det}(J)$ and $\mathrm{tr}(J)$, respectively.
Fixed point $\big(0, 0 \big)$ is a saddle, since 
$\mathrm{det}(J) = a^2 -1 < 0$. The saddle quantity from 
Andronov-Leontovich theorem in~\ref{HomoclinicTheorem} is 
given by $\sigma_0 = \lambda_1 + \lambda_2 = (a - 1) + (a + 1) = 2 a$, 
were $\lambda_1$ and $\lambda_2$ are the saddle eigenvalues. 
The alpha loop is stable from the inside provided $\sigma_0 < 0$, 
implying $a < 0$. It then follows that $\big(2/9 a, -2/3\big)$ 
is an unstable spiral, and 
$\big(a^{-1} (1 - a^{-2}), \, -1 + a^{-2} \big)$ a stable node.
\end{proof}

\noindent
A representative phase plane diagram of 
system~(\ref{ansatza})--(\ref{ansatzb}) is shown in
Figure~\ref{fig:alpha}(b). Note that a part of the positive alpha 
branch $x_1^{+}$ is a heteroclinic orbit connecting the saddle 
and the node. The distance between the saddle and the node is 
given by $d(a) = a^{-3} (1-a^2) \sqrt{1 + a^2}$, so 
that $\lim_{a \to 0} d(a) = + \infty$ 
and $\lim_{a \to -1} d(a) = 0$, i.e. increasing $a$ 
increases length of the heteroclinic orbit by sliding 
the node along $x_1^{+}$. 

System~(\ref{ansatza})--(\ref{ansatzb}) satisfies the first 
three conditions of Andronov-Leontovich theorem 
in~\ref{HomoclinicTheorem}. In order to satisfy the final 
condition, a set of perturbations must be found that 
destroy the alpha loop in a generic way, and this is ensured 
by the Melnikov condition~(\ref{Melnikov}). The bifurcation 
parameter controlling the existence of the alpha loop is 
denoted as $\alpha \in \mathbb{R}$. Note that 
$\mathcal{P}(\mathbf{x}; \, a)$ perturbed by a function of the 
form $\alpha \nabla H(x_1,x_2) = \alpha (-2 x_1, 2 x_2 + 3 x_2^2)$ 
satisfies the Melnikov condition~\cite{Sand}, 
but $\mathcal{P}(\mathbf{x}; \, a) + \alpha \nabla H(x_1,x_2)$ has 
three terms dependent on $\alpha$. In the following theorem, 
a simpler set of perturbations is found, introducing only 
one $\alpha$ dependent term in system~(\ref{ansatza})--(\ref{ansatzb}).
 
\begin{theorem} 
\label{Melni}
If a perturbation of the form $(\alpha f(x_1),0)^T$, where 
$\alpha \in \mathbb{R}$, is added to the RHS of 
system~$(\ref{ansatza})$--$(\ref{ansatzb})$, and if $f(x_1)$ 
is an odd function, $f(-x_1) = - f(x_1)$, and $f(x_1) \ne 0$, 
$\forall x_1 \in \lbrack -2 \sqrt{3}/9,0)$, then the 
perturbed system undergoes a supercritical homoclinic 
bifurcation in a \emph{generic way} as $\alpha$ is varied 
in the neighbourhood of zero. 
\end{theorem}

\begin{proof}
Consider the perturbed version of 
system~(\ref{ansatza})--(\ref{ansatzb}):
\begin{align}
\frac{\mathrm{d} x_1}{\mathrm{d}  t} 
& = \mathcal{P}_1(\mathbf{x}; \, a, \alpha) 
= a x_1 +  x_2 + \frac{3}{2} a x_1 x_2  + \frac{3}{2} x_2^2 + \alpha f(x_1),
\label{ansatzz2a} 
\\
\frac{\mathrm{d}  x_2}{\mathrm{d}  t}  
& = \mathcal{P}_2(\mathbf{x}; \, a) 
=  x_1 + a x_2 + a x_2^2.
\label{ansatzz2b}
\end{align}
Melnikov integral~(\ref{Melnikov}) 
for system~(\ref{ansatzz2a})--(\ref{ansatzz2b}) 
is given by
$$
M(0) 
=  - \int_{- \infty}^{+ \infty} \varphi(t) f(x_1)  
\mathcal{P}_2(x_1, x_2; \, a) \, \mathrm{d} t. 
$$
Using~(\ref{ansatzz2b}), we have  
$\mathcal{P}_2(x_1, x_2; \, a) \mathrm{d} t = \mathrm{d} x_2$.
Thus we can express $M(0)$ in terms of $x_2$ 
as follow:
\begin{align*}
M(0) 
& = 
- 
\int_{- \infty}^{0} \varphi(t) f(x_1) 
\mathcal{P}_2(x_1, x_2; \, a) \, \mathrm{d} t 
- 
\int_{0}^{+ \infty} 
\varphi(t) f(x_1) \mathcal{P}_2(x_1, x_2; \, a) 
\, \mathrm{d} t \\
& = 
\int_{-1}^{0} 
\varphi \big(t^{+}(x_2) \big) 
f \big(x_2\sqrt{1 + x_2} \big) 
\, \mathrm{d} x_2
-
\int_{-1}^{0}    
\varphi \big(t^{-}(x_2) \big) f \big(\!-\!x_2\sqrt{1 + x_2})  
\, \mathrm{d} x_2,
\end{align*}
where $t^{+}(x_2)$ (resp. $t^{-}(x_2)$) is the dependence of 
time $t$ on $x_2$ along the positive (resp. negative) alpha 
branch for the trajectory which is at point $(x_1,x_2) = (0,-1)$ 
at time $t=0$ (for $\alpha=0$). 
Since $f$ is odd and $\varphi(t^{\pm}) > 0$, we deduce
$$
M(0) =  
\int_{-1}^{0} 
\Big[ 
\varphi \big(t^{+}(x_2) \big) 
+ 
\varphi \big(t^{-}(x_2) \big)
\Big] 
f \left(x_2 \sqrt{1 + x_2} \right)
\, \mathrm{d} x_2 \; \; \ne \; \; 0. 
$$
\end{proof}

\noindent
For further simplicity of~(\ref{ansatzz2a})--(\ref{ansatzz2b}),
function $f(x_1)$ is set to $f(x_1) = x_1$ in the rest of this
paper.

\subsection{Step \emph{(2)}: construction of kinetic function
$\mathcal{K}(\bar{x}; \, \bar{k})$} 
\label{sec:step2}
The RHS of system~(\ref{ansatzz2a})--(\ref{ansatzz2b}), $\mathbf{\boldsymbol{\mathcal{P}}}(\mathbf{x}; \, a, \alpha)$, 
is a kinetic function. However, the alpha loop, which is the 
region of interest, is not located in the nonnegative orthant. 
In order to position the loop into the positive orthant, we will 
apply affine transformations. First, we show that 
system~(\ref{ansatzz2a})--(\ref{ansatzz2b}) with the homoclinic 
orbit in nonnegative orthant is nonkinetic under all translation 
transformations for $a \in (-1,0)$, $\alpha \in \mathbb{R}$.

\begin{lemma}\label{auxlemma}
Function $\mathbf{\boldsymbol{\mathcal{P}}}(\mathbf{x}; \, a, \alpha)$, 
given by the \emph{RHS} of $(\ref{ansatzz2a})$--$(\ref{ansatzz2b})$, 
is nonkinetic under all translation transformations 
$\Psi_{\mathcal{T}}$, for $a \in (-1,0)$ and $\alpha \in \mathbb{R}$, 
given the condition that the homoclinic orbit is mapped 
into $\mathbb{R}^2_{>}$.
\end{lemma}

\begin{proof}
Let us apply the translation transformation $\Psi_{\mathcal{T}}$ 
(see Definition~\ref{def:translation}), 
$\boldsymbol{\mathcal{T}} = (\mathcal{T}_1,\mathcal{T}_2) \in \mathbb{R}^2$, 
on function $\boldsymbol{\mathcal{P}}(\mathbf{x}; \, a, \alpha)$, 
given by the \emph{RHS} of $(\ref{ansatzz2a})$--$(\ref{ansatzz2b})$, 
resulting in:
\begin{align}
(\Psi_{\mathcal{T}} \mathcal{P}_1)(\mathbf{\bar{x}}; \, \mathbf{\bar{k}})& 
= \bar{k}_0^1 + \bar{k}_1^1 \bar{x}_1 + \bar{k}_2^1 \bar{x}_2 
+ \bar{k}_{1 2}^1 \bar{x}_1 \bar{x}_2 + \bar{k}_{2 2}^1 (\bar{x}_2)^2 , \nonumber \\
(\Psi_{\mathcal{T}} \mathcal{P}_2)(\mathbf{\bar{x}}; \, \mathbf{\bar{k}})& 
=  \bar{k}_{0}^{2} + \bar{k}_{1}^2 \bar{x}_1 
+ \bar{k}_{2}^2  \bar{x}_2 + \bar{k}_{2 2}^2 (\bar{x}_2)^2, \label{zerorot}
\end{align}
with $\mathbf{\bar{x}} = \mathbf{x} + \boldsymbol{T}$, 
and coefficients 
$\mathbf{\bar{k}} = \mathbf{\bar{k}}(a, \alpha, \boldsymbol{\mathcal{T}})$ 
given by
\begin{align}
\bar{k}_0^1  & = 
\frac{1}{2} 
\big(3 
(\mathcal{T}_2 - \frac{2}{3})(a \mathcal{T}_1 + \mathcal{T}_2) -2 \alpha \mathcal{T}_1 
\big), \nonumber \\
\bar{k}_0^2 & = 
-\mathcal{T}_1 + a \mathcal{T}_2 (\mathcal{T}_2 - 1), \nonumber  \\
\bar{k}_1^1  & = 
-\frac{3}{2} a (\mathcal{T}_2 - \frac{2}{3}) + \alpha,  \; \; \; \; \; \;  
\bar{k}_1^2 = 1,  \nonumber \\
\bar{k}_2^1  &=  
1 - \frac{3}{2} (a \mathcal{T}_1 + 2 \mathcal{T}_2), \; \; \; \; \; \;  
\bar{k}_2^2 = - 2 a \big(\mathcal{T}_2 - \frac{1}{2} \big),  \nonumber \\
\bar{k}_{1 2}^1 & = 
\frac{3}{2} a, \; \; \; \; \; \; \; \bar{k}_{2 2}^1   
= \frac{3}{2}, \; \; \; \; \; \; \; \bar{k}_{2 2}^2 = a. \label{h5} 
\end{align}
Consider the point $\mathbf{x_0} = (0,-1)$, which is on 
the alpha loop. It is mapped by $\Psi_{\mathcal{T}}$ to 
the point $\mathbf{\bar{x}_0} = (\mathcal{T}_1,-1 + \mathcal{T}_2)$. 
Requiring that the alpha loop is mapped to $\mathbb{R}^2_{>}$ 
implies that we must have $\mathbf{\bar{x}_0} \in \mathbb{R}_{>}^2$, 
so that the following set of constraints 
(see Definition~\ref{def:constraints}) must be satisfied:
\begin{align}
\Phi = \{ \mathcal{T}_1 > 0, \mathcal{T}_1 > 1 \}. \label{tinequality}
\end{align}
Using the fact that $a \in (-1,0)$ and the 
constraints~(\ref{tinequality}), it follows that $\bar{k}_0^2$ 
from~(\ref{h5}) is negative, $\bar{k}_0^2< 0$. Thus, 
$(\Psi_{\mathcal{T}} \boldsymbol{\mathcal{P}})(\mathbf{\bar{x}}; \, \mathbf{\bar{k}})$ 
has a cross-negative term, and the statement of the theorem follows.
\end{proof}
\noindent
One can also readily prove that 
$\mathbf{\boldsymbol{\mathcal{P}}}(\mathbf{x}; \, a, \alpha)$ 
is nonkinetic under all affine transformations for $|a| \ll 1$, 
and $|\alpha| \ll 1$. 
Thus, in the next two sections, we follow Step (2), case (b), 
in Algorithm~\ref{tabinv}, applying transformations 
$\Psi_{\mathcal{X}}$ and $\Psi_{\mathrm{QSSA}}$ on the 
kinetic function 
$(\Psi_{\mathcal{T}} \boldsymbol{\mathcal{P}})(\mathbf{\bar{x}}; 
\, \mathbf{\bar{k}})$ 
given by~(\ref{zerorot}). We require the following 
conditions to be satisfied:
\begin{align}
a & \in (-1,0), \; \; \; |\alpha| \ll 1,  \nonumber \\
\Phi & = 
\left\{ \mathcal{T}_1 > \frac{2 \sqrt{3}}{9}, \mathcal{T}_2 > 1 \right\},
\label{constraintz}
\end{align}
with the set of constraints $\Phi$ ensuring that 
the homoclinic orbit is in $\mathbb{R}^2_{>}$. The 
reason for requiring $|\alpha| \ll 1$ is that then 
the following results are more readily derived, and 
the condition is sufficient for studying system~(\ref{zerorot}) 
near the bifurcation point $\alpha = 0$. 
A representative phase plane diagram corresponding 
to the ODE system with RHS~(\ref{zerorot}) is shown in 
Figure~\ref{fig:times}(a), 
with fixed points, the alpha curve, and the vector field 
denoted as in Figure~\ref{fig:alpha}(b), and with 
the red segments on the axes corresponding to the 
phase plane regions where the cross-negative effect 
exists (see Definition~\ref{def:cne}).

\begin{figure}[t]
\centerline{
\hskip -2mm
\includegraphics[width=0.46\columnwidth]{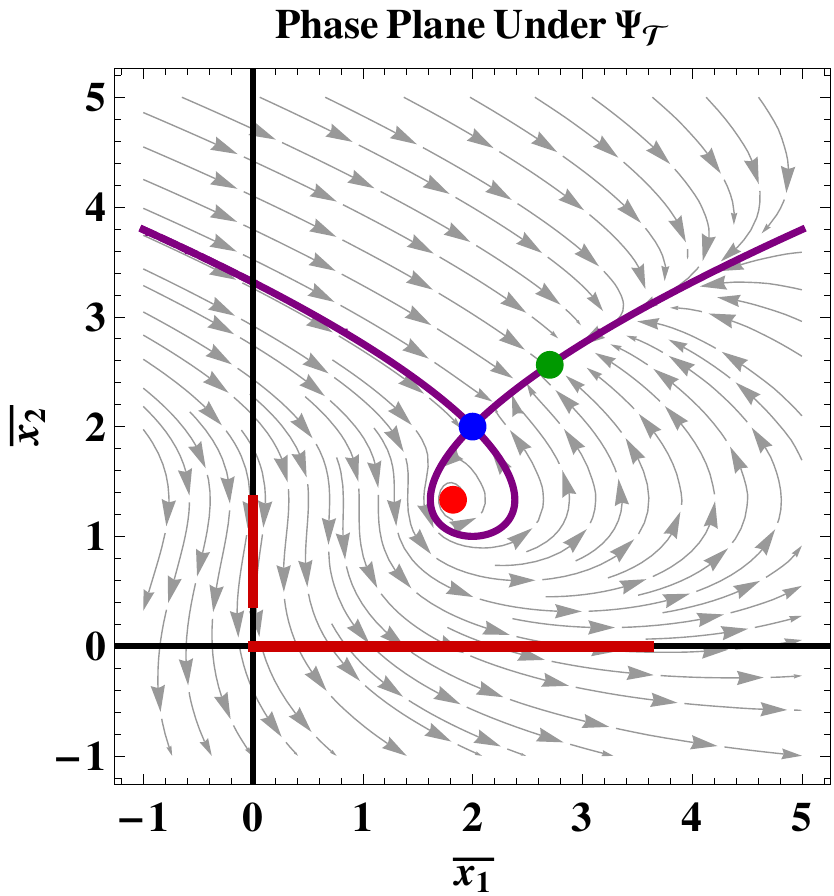}
\hskip 6mm
\includegraphics[width=0.46\columnwidth]{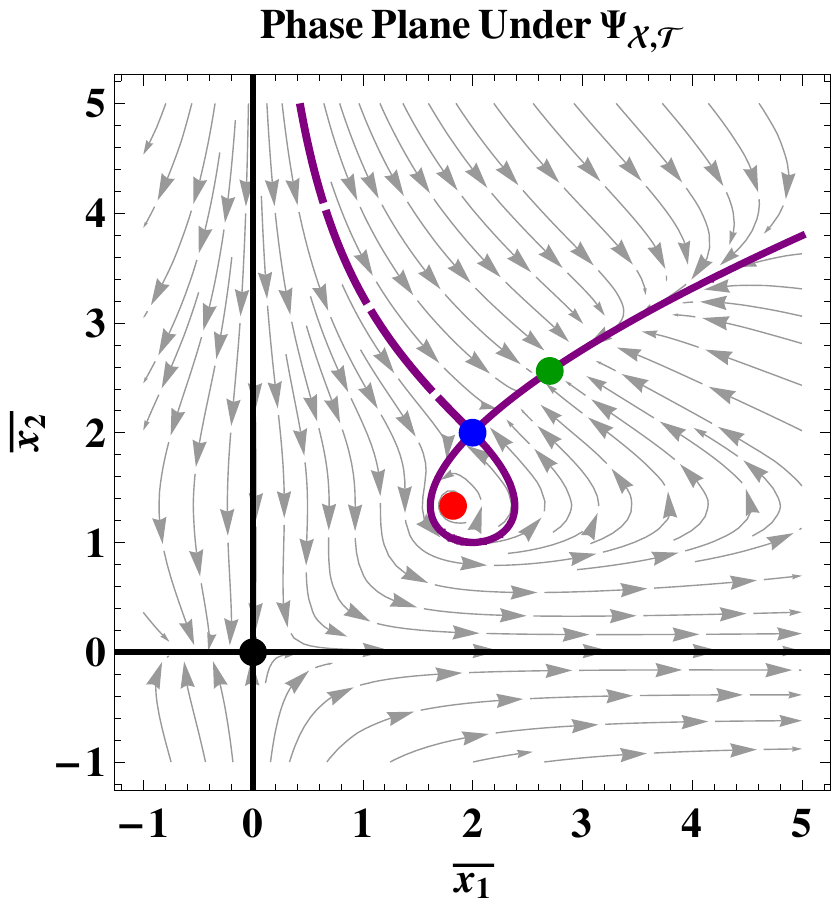}
}
\vskip -6.5cm
\leftline{\hskip 1mm (a) \hskip 5.95cm (b)}
\vskip 6.0cm
\centerline{
\hskip -2mm
\includegraphics[width=0.46\columnwidth]{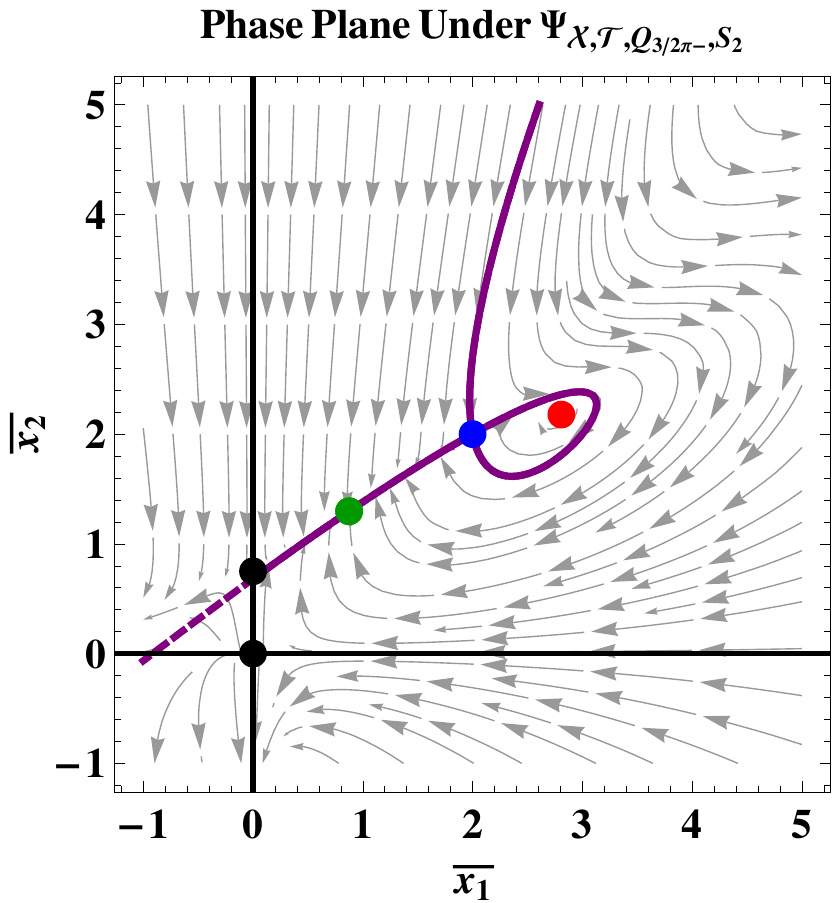}
\hskip 6mm
\includegraphics[width=0.46\columnwidth]{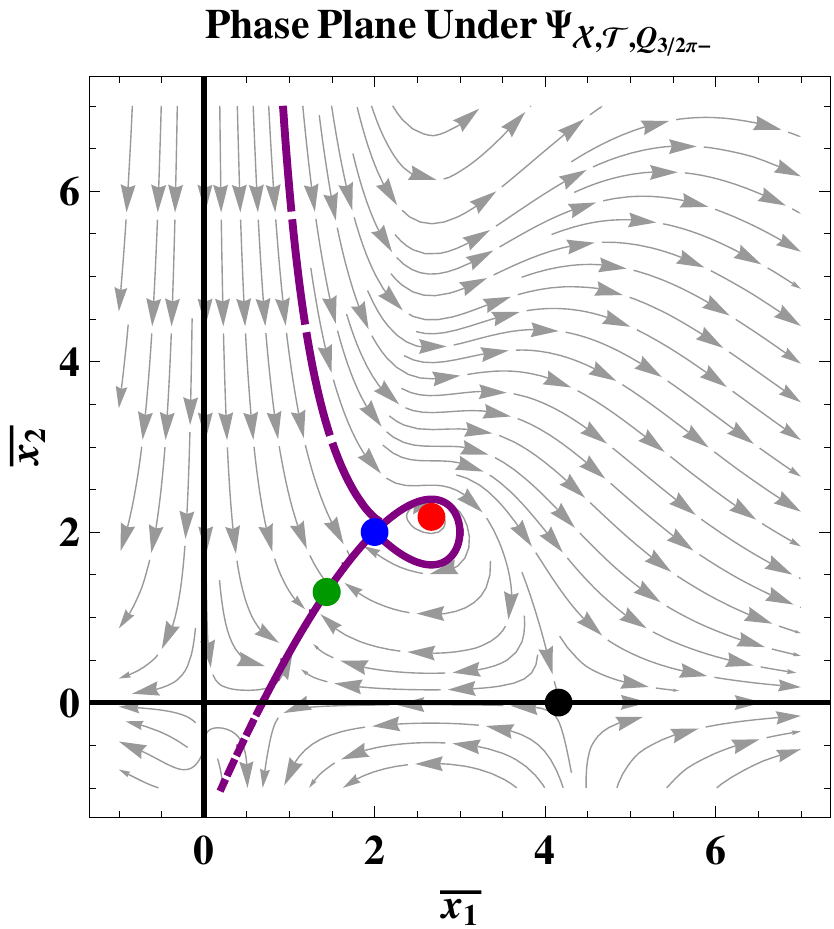}
}
\vskip -6.5cm
\leftline{\hskip 1mm (c) \hskip 5.95cm (d)}
\vskip 6.0cm
\caption{ 
{\it Phase plane diagrams of} {\rm (a)} {\it ODE system with RHS~$(\ref{zerorot})$;}
{\rm (b)} {\it ODE system with RHS~$(\ref{zerorot2})$;}
and {\rm (c)} {\it ODE system with RHS~$(\ref{zerorot3})$.} 
{\it The stable node, the saddle, and the unstable spiral are 
represented as the green, blue and red dot, respectively. 
The alpha curve is shown in purple, and the vector field 
as gray arrows. On each plot it is indicated which kinetic 
transformation is applied to system~$(\ref{ansatzz2a})$--$(\ref{ansatzz2b}).$
Red segments of the phase plane axes in} {\rm (a)} 
{\it denote the regions where the
cross-negative effect exists. In} {\rm (b)} {\it and} 
{\rm (c)}, {\it boundary fixed points 
are represented as the black dots, purple curves with a coarser 
dashing represent the saddle manifolds that asymptotically approach 
an axis, while with a finer dashing those that are outside 
of $\mathbb{R}_{\ge}^2$.} \hfill\break
{\rm (d)} {\it Phase plane diagram of a system for which $x$-factorable
transformation significantly globally influences the phase curves 
in such a way that a pure translation cannot resolve the artefacts. 
For more details, see the text. The parameters are fixed to 
$a = -0.8$, $\alpha = 0$, $\mathcal{T}_1 = 2$, $\mathcal{T}_2 = 2$.}}
\label{fig:times}
\end{figure}

\subsubsection{X-factorable transformation} 
\label{sec:xftc}

Let us apply the $x$-factorable transformation 
$\Psi_{\mathcal{X}}$ on system~(\ref{zerorot}). 
Letting 
$\Psi_{\mathcal{X}, \mathcal{T} } \equiv \Psi_{\mathcal{X}} \circ \Psi_{\mathcal{T}}$, 
the resulting kinetic function 
$\boldsymbol{\mathcal{K}}_{\mathcal{X}, \mathcal{T}}(\mathbf{\bar{x}}; 
\, \mathbf{\bar{k}}) \equiv 
(\Psi_{\mathcal{X}, \mathcal{T}} \boldsymbol{\mathcal{P}}) (\mathbf{\bar{x}}; \, 
\mathbf{\bar{k}})$ is given by
\begin{align}
\mathcal{K}_{1, \mathcal{X}, \mathcal{T}}(\mathbf{\bar{x}}; \, \mathbf{\bar{k}})& = 
\bar{x}_1 ( \bar{k}_0^1 + \bar{k}_1^1 \bar{x}_1 + \bar{k}_2^1 \bar{x}_2 + 
\bar{k}_{1 2}^1 \bar{x}_1 \bar{x}_2 + \bar{k}_{2 2}^1 (\bar{x}_2)^2) , \nonumber \\
\mathcal{K}_{2, \mathcal{X}, \mathcal{T}}(\mathbf{\bar{x}}; \, \mathbf{\bar{k}}) & = 
\bar{x}_2 (\bar{k}_{0}^{2} + \bar{k}_{1}^2 \bar{x}_1 + 
\bar{k}_{2}^2  \bar{x}_2 + \bar{k}_{2 2}^2 (\bar{x}_2)^2).\label{zerorot2}
\end{align}

\begin{theorem}\label{XFactor}
ODE systems with RHSs~$(\ref{zerorot})$ and~$(\ref{zerorot2})$ are topologically 
equivalent in the neighbourhood of the fixed points in $\mathbb{R}_{>}^2$, with 
the homoclinic orbit in $\mathbb{R}^2_{>}$, and a saddle at the origin 
being the only boundary fixed point in $\mathbb{R}_{\ge}^2$, if:
\begin{align}
a & \in (-1,0), \; \; \; |\alpha| \ll 1,  \nonumber \\
\Phi & 
= \left\{ \mathcal{T}_1  > \frac{2 \sqrt{3}}{9}, 
\mathcal{T}_2 \in \left(\mathrm{max}(1,-a \mathcal{T}_1), 
\frac{2}{3} + \frac{8}{3} a^{-2} (3-a^2) (a + 4 \mathcal{T}_1) \right) 
\right\}.\label{Xonecond}
\end{align}
\end{theorem}

\begin{proof}
Let us assume $\alpha = 0$. 

\noindent
From statement (i) of Theorem~\ref{Xfactdet} it follows that 
the saddle fixed point of~(\ref{zerorot}) is preserved under 
$\Psi_{\mathcal{X}}$. Denoting the node and spiral fixed points 
of~(\ref{zerorot}) by 
$\mathbf{\bar{x}}^*_{\mathrm{nd}}$ and $\mathbf{\bar{x}}^*_{\mathrm{sp}}$, 
respectively, one finds that the Jacobian is given by:
\begin{align}
J_{\mathcal{T}}|_{\mathbf{\bar{x}} 
= \mathbf{\bar{x}}^*_{\mathrm{nd}}} &= \left(\begin{array}{cc}
a + \frac{3}{2} a^{-1}(1-a^2) & \; \; \frac{1}{2} a^{-2}(3-a^2) 
\\
1  & \; \;  a^{-1}(2-a^2) 
\end{array}
\right)   \implies   \mathrm{sign}(J_{\mathcal{T}}|_{\mathbf{\bar{x}} = \mathbf{\bar{x}}^*_{\mathrm{nd}}}) = \left(\begin{array}{cc}
- & \; \; +
\\
+  & \; \;  -
\end{array}
\right), \nonumber \\
J_{\mathcal{T}}|_{\mathbf{\bar{x}} 
= \mathbf{\bar{x}}^*_{\mathrm{sp}}} &= \left(\begin{array}{cc}
0 & \; \; - \frac{1}{3} (3-a^2) 
\\
1  & \; \;  -\frac{1}{3} a
\end{array}
\right) \;\;\;\;\;\;\;\;\;\;\;\;\;\;\;\;\;\;\;\;\;\;\;\;\;\;\;\; \implies  \mathrm{sign}(J_{\mathcal{T}}|_{\mathbf{\bar{x}} = \mathbf{\bar{x}}^*_{\mathrm{sp}}}) = \left(\begin{array}{cc}
0 & \; \; -
\\
+  & \; \;  +
\end{array}
\right).
\end{align}
Conditions (ii) and (iii) of Theorem~\ref{Xfactdet} are both 
satisfied for the node, but only condition (ii) is satisfied 
for the spiral. The condition for the spiral to remain 
invariant is obtained by demanding 
$\mathrm{disc}(J_{\mathcal{X}, \mathcal{T}}|_{\mathbf{\bar{x}} 
= \mathbf{\bar{x}}^*_{\mathrm{sp}}}) < 0$, where 
$J_{\mathcal{X}, \mathcal{T}}$ is the Jacobian of~(\ref{zerorot2}), 
and, taking into consideration~(\ref{constraintz}), this leads to
\begin{align}
\mathcal{T}_2 & < 
\frac{2}{3} + \frac{8}{3} a^{-2} (3-a^2) (a + 4 \mathcal{T}_1).
\end{align}
Boundary fixed points are given by 
$(0,0)$, $(\mathcal{T}_1 + a^{-1} \mathcal{T}_2,0)$, 
$(0, 1/2(1 \pm \sqrt{1 + 4 a^{-1} \mathcal{T}_1}) + \mathcal{T}_2)$. 
The second fixed point can be removed from the nonnegative quadrant 
by demanding $\mathcal{T}_2 > - a \mathcal{T}_1$, while the pair 
of the last ones always has nonzero imaginary part due 
to~(\ref{constraintz}). Statement (iv) of Theorem~\ref{Xfactdet} 
implies that the eigenvalues at the origin are given by 
$\lambda_1 = k_0^1 =3/2(\mathcal{T}_2 - 2/3) (a \mathcal{T}_1 
+ \mathcal{T}_2) > 0$ and 
$\lambda_2 = k_0^2 = -\mathcal{T}_1 + a \mathcal{T}_2 (\mathcal{T}_2 - 1) < 0$, 
so origin is a saddle fixed point. 

As $\alpha$ can be chosen arbitrarily close to zero, 
and as  
$\boldsymbol{\mathcal{K}}_{\mathcal{X}, \mathcal{T}}(\mathbf{\bar{x}}; 
\, \mathbf{\bar{k}})$ is a continuous function of $\alpha$, 
the theorem holds for sufficiently small $|\alpha| \ne 0$, as well.
\end{proof}
\noindent
A representative phase plane diagram corresponding to the ODE system
with RHS~(\ref{zerorot2}) is shown in Figure~\ref{fig:times}(b), where 
the saddle fixed point at the origin is shown as the black dot. It can 
be seen that one of the stable manifolds of the nonboundary saddle, 
represented as a dashed purple curve, approaches $x_2$-axis asymptotically, 
instead of crossing it as in Figure~\ref{fig:times}(a). 

The homoclinic orbit of the ODE system with RHS~(\ref{zerorot2}) is 
positioned below the node in the phase plane. Suppose the relative 
position of the stable sets is reversed by, say, applying an 
improper orthogonal matrix with the angle fixed to $3\pi/2$, 
$\Psi_{Q_{3\pi/2 -}}$, with a representative phase plane shown 
in Figure~\ref{fig:times}(d). In this case, one can straightforwardly 
show that the boundary fixed point given by 
$\big( \mathcal{T}_1 + 1/2 (1 + \sqrt{1 - 4 a^{-1} \mathcal{T}_2}), 0 \big)$, 
shown as the black dot in Figure~\ref{fig:times}(d), 
cannot be removed from $\mathbb{R}_{\ge}^2$, and is always 
a saddle. The same conclusions are true for other 
similar configurations of the stable sets obtained 
by rotations only. This demonstrates that $x$-factorable 
transformation can produce boundary artefacts that have 
a significant \emph{global} influence on the phase curves, 
that cannot be eliminated by simply translating a region of 
interest sufficiently far away from the axes. In order to 
eliminate the particular boundary artefact, a shear 
transformation may be applied. Consider applying 
$\Psi_{\mathcal{X}, \mathcal{T}, \mathcal{Q}_{3\pi/2 -}, \mathcal{S}_2}
= \Psi_{\mathcal{X}} \circ \Psi_{\mathcal{T}} \circ \Psi_{\mathcal{Q}_{3\pi/2 -}}
\circ \Psi_{\mathcal{S}_2}$ 
on~(\ref{ansatzz2a})--(\ref{ansatzz2b}), where 
\begin{align}
S_{2} & = \left(\begin{array}{cc}
1 & \; \; 0
\\
-a  & \; \; 1
\end{array}
\right),
\end{align}
and 
$\mathcal{T}_1 = \mathcal{T}_2 \equiv \mathcal{T} \in \mathbb{R}$, 
for simplicity, leading to
\begin{align}
\mathcal{K}_{n, \mathcal{X}, \mathcal{T}, \mathcal{Q}_{3\pi/2 -}, 
\mathcal{S}_2}(\mathbf{\bar{x}}; \, \mathbf{\bar{k}}) & = 
\bar{x}_n ( \bar{k}_{0}^{n} + \bar{k}_{1}^n \bar{x}_1 + \bar{k}_{2}^n \bar{x}_2 
\nonumber \\
 & + \bar{k}_{1 1}^n (\bar{x}_1)^2 + \bar{k}_{1 2}^n \bar{x}_1 \bar{x}_2 
 + \bar{k}_{2 2}^n (\bar{x}_2)^2 ), \; \; \; n \; = \; 1, 2,  
 \label{zerorot3} 
\end{align}
where the coefficients 
$\mathbf{\bar{k}} = \mathbf{\bar{k}}(a, \alpha, \mathbf{\mathcal{T}})$ 
are given by:
\begin{align}
\bar{k}_0^1  & = 
\frac{1}{2} \mathcal{T} \big(-2 + a (2 \alpha + \mathcal{T} + 
a (2 + 5 \mathcal{T}) +  4 \mathcal{T} a^2) \big), 
\nonumber \\
 \bar{k}_0^2 & =-\frac{1}{2} \mathcal{T}  
 \big(2 + 2 \alpha + 3 \mathcal{T} + a(4 + 9 \mathcal{T} 
 + 6 \mathcal{T} a ) \big), 
 \nonumber \\
\bar{k}_1^1  & =-\frac{1}{2} \mathcal{T} a (2 + 5 a),  
\; \; \; \; \; \;  \bar{k}_1^2 =  
1 + \frac{9}{2} \mathcal{T} (\frac{2}{3} + a), 
\nonumber  \\
\bar{k}_2^1  &= 
1 - \frac{1}{2} a \big(2 \alpha + a (2 + 5 \mathcal{T}  
+ 8 \mathcal{T} a ) \big), \; \; \; \; \; \;  \bar{k}_2^2 = 
\alpha + a \big( 2 + \frac{9}{2}\mathcal{T} + 6 \mathcal{T} a \big),  
\nonumber \\
\bar{k}_{1 1}^1 & = \frac{1}{2} a, \; \;  \bar{k}_{1 1}^2   
= - \frac{3}{2}, \; \;  \bar{k}_{1 2}^1 
= \frac{5}{2} a^2, \; \;  \bar{k}_{1 2}^2 
= - \frac{9}{2} a, \; \;  \bar{k}_{2 2}^1 = 2 a^3, \; \;  \bar{k}_{2 2}^2 
= -3 a^2,\label{x2coeff}
\end{align}
together with $a \in (-1,0)$, $|\alpha| \ll 1$ and $\Phi = \{ \mathcal{T} >  1 \}$.

\begin{theorem}\label{XFactor2}
ODE systems with RHSs~$(\ref{ansatzz2a})$--$(\ref{ansatzz2b})$ 
and~$(\ref{zerorot3})$ are topologically equivalent in the 
neighbourhood of the fixed points in $\mathbb{R}_{>}^2$, with 
the homoclinic orbit in $\mathbb{R}^2_{>}$, and a saddle at the 
origin and a saddle with coordinates 
$(0, 1/2 \mathcal{T} a^{-1} (1 + 2 a))$ being the only boundary fixed 
points in $\mathbb{R}_{\ge}^2$, if:
\begin{align}
a & \in \left(-1,-\frac{1}{2} \right), \; |\alpha| \ll 1, \; \qquad \mbox{and} \qquad 
\nonumber \\ 
\Phi  &= \{\mathcal{T} > \mathrm{max}(1, 2 a^{-2} (1 - a^2)), \mathcal{T} 
< 2 a^{-1} (1 + 4 a)^{-1} (1 - a) \}.\label{conditions3}
\end{align}
\end{theorem}

\begin{proof}
Following the same procedure as in the proof of Theorem~\ref{XFactor}, 
and noting that the saddle, node and spiral fixed points are 
given by $\big( \mathcal{T}, \mathcal{T} \big)$, 
$\big( \mathcal{T} -2 a^{-2} (1-a^2), \mathcal{T} + a^{-3} (1-a^2) \big)$, 
and $\big( \mathcal{T} + 2/9 (3 + a^2), \mathcal{T} -  2/9 a \big)$, 
respectively, while the five boundary fixed points are 
$(0,0)$, 
$(\mathcal{T} + 1/2 
\big( 
5\mathcal{T} a \pm \sqrt{\mathcal{T} (8 a^{-1}(1-a^2) + 9 \mathcal{T} a^2)}\big),0)$, 
$(0, 1/2 \mathcal{T} a^{-1} (1 + 2 a))$, 
$(0, a^{-1} \big(2/3 + \mathcal{T} (1 + a) \big))$, 
one finds~(\ref{conditions3}).
\end{proof}
\noindent
Note that as $\alpha \to -\frac{1}{2}$, the only boundary fixed point 
in $\mathbb{R}_{\ge}^2$ is a saddle at the origin, and it is connected 
via a heteroclinic orbit to the saddle in $\mathbb{R}_{>}^2$. 
A representative phase plane diagram corresponding to the ODE system
with RHS~(\ref{zerorot3}) is shown in Figure~\ref{fig:times}(c).

While systems~(\ref{zerorot2}) and~(\ref{zerorot3}) contain 
specific variations of the specific homoclinic orbit given 
by~(\ref{loop}), they, nevertheless, indicate possible phase 
plane topologies of the kinetic equations containing homoclinic 
orbits of shapes similar to~(\ref{loop}). With a fixed shape 
and orientation of a homoclinic loop which is similar 
to~(\ref{loop}), three possible orientations of a corresponding 
saddle manifold in $\mathbb{R}_{>}^2$ are: it may extend 
in $\mathbb{R}_{>}^2$ without asymptotically approaching 
a phase plane axis, it may asymptotically approach an axis, 
or it may cross an axis at a fixed point. In Figure~\ref{fig:times}(b), 
a combination of the first and second orientation is displayed, 
while in Figure~\ref{fig:times}(c) of the first and third orientation. 

\subsubsection{The quasi-steady state transformation}
In Lemma~\ref{auxlemma}, it was demonstrated that 
system~(\ref{ansatzz2a})--(\ref{ansatzz2b}) has at least 
one cross-negative term under translation transformations. 
It can be readily shown that system~(\ref{ansatzz2a})--(\ref{ansatzz2b}) 
in fact has minimally two cross-negative terms under 
translation transformations, $\bar{k}_2^1 < 0$ and 
$\bar{k}_0^2 < 0$, and this is the case when $a \in (-1,0)$, 
$\Phi = \{\mathcal{T}_1 \in (2 \sqrt{3}/9, -\mathcal{T}_2 a), \mathcal{T}_2 > 1\}$. 
Let us apply a quasi-steady state transformation 
$\Psi_{\mathrm{QSSA}}$ on system~(\ref{zerorot}) that 
eliminates the two cross-negative terms, i.e. two new 
variables are introduced, $\bar{y}_1, \bar{y}_2 \in \mathbb{R}_{>}^{2}$, 
and we take $p_1(\bar{x}) = p_2(\bar{x}) = 1$, in 
Definition~\ref{def:QSSA}. Letting 
$\Psi_{\mathrm{QSSA}, \mathcal{T}} \equiv \Psi_{\mathrm{QSSA}} \circ \Psi_{\mathcal{T}}$, 
the resulting kinetic function  
$\mathcal{K}_{\mathrm{QSSA}, \mathcal{T}}(\{ \bar{x}, \bar{y} \}; 
\, \bar{k}, \omega, \mu) \equiv \big(\Psi_{\mathrm{QSSA, \mathcal{T}}} \mathcal{P}) 
(\{ \bar{x}, \bar{y} \}; \, \bar{k}, \omega, \mu)$ is given 
by
\begin{align}
\mathcal{K}_{x_1, \mathrm{QSSA}, 
\mathcal{T}}(\{\bar{x}_1, \bar{x}_2, \bar{y}_1\}; \, \bar{k}, \omega_1) & 
= \bar{k}_{0}^{1} + \bar{k}_{1}^1 \bar{x}_1 +  
\bar{k}_{2}^1 \omega_1^{-1} \bar{x}_1 \bar{x}_2 \bar{y}_1 + 
\bar{k}_{1 2}^1 \bar{x}_1 \bar{x}_2 + \bar{k}_{2 2}^1 (\bar{x}_2)^2 , 
\nonumber \\
\mathcal{K}_{x_2, \mathrm{QSSA}, \mathcal{T}}(\{\bar{x}_1, \bar{x}_2, \bar{y}_2\}; 
\, \bar{k}, \omega_2) & 
=  \bar{k}_{0}^{2}  \omega_2^{-1} \bar{x}_2 \bar{y}_2 + \bar{k}_{1}^2 \bar{x}_1 
+ \bar{k}_{2}^2 \bar{x}_2 + \bar{k}_{2 2}^2 (\bar{x}_2)^2, 
\nonumber\\
\mathcal{K}_{y_1, \mathrm{QSSA}, \mathcal{T}}(\{\bar{x}_1, \bar{y}_1 \}; \, \omega_1, \mu) & 
= \mu^{-1} (\omega_1 - \bar{x}_1 \bar{y}_1), 
\nonumber\\
\mathcal{K}_{y_2, \mathrm{QSSA}, \mathcal{T}}(\{\bar{x}_2, \bar{y}_2 \}; \, \omega_2, \mu) & 
=  \mu^{-1} (\omega_2 - \bar{x}_2 \bar{y}_2),
\label{QSSAs}
\end{align}
with $\bar{x}_n (0) > 0$, $\omega_n > 0$, $n = 1, 2$, and $\mu \to 0$.

In~(\ref{QSSAs}), $\lim_{x_n \to 0} \lim_{\mu \to 0} y_n = + \infty$, $n = 1,2$, 
and a geometrical implication is that, say, the saddle manifold crossing the $x_2$-axis 
in Figure~\ref{fig:times}(a), instead asymptotically approaches the $y_1$-axis. 
The asymptotic behavior of the saddle manifolds is achieved by the additional 
(boundary) fixed points in~(\ref{zerorot2}) displayed in Figure~\ref{fig:times}(b), 
and by additional phase space dimensions in~(\ref{QSSAs}).

Note that a composition an $x$-factorable and a quasi-steady state 
transformation may be used to make~(\ref{zerorot}) kinetic. 
For example, one may eliminate the cross-negative term 
$\bar{k}_2^1$ in~(\ref{zerorot}) by using the $x_1$-factorable 
transformation $\Psi_{\mathcal{X}_1}$, and the cross-negative 
term $\bar{k}_0^2$ by using an appropriate $\Psi_{\mathrm{QSSA}}$. 
An example of a kinetic function obtained by a transformation 
of the form $\Psi_{\mathrm{QSSA}, \mathcal{X}_1, \mathcal{T}}$ 
is given by
\begin{align}
\mathcal{K}_{x_1,\mathcal{\mathrm{QSSA}},\mathcal{X}_1,\mathcal{T}}(\bar{x}; \, \bar{k}) & 
= \bar{x}_1 \big( \bar{k}_{0}^{1} + \bar{k}_{1}^1 \bar{x}_1 +  
\bar{k}_{2}^1 \bar{x}_2 + \bar{k}_{1 2}^1 \bar{x}_1 \bar{x}_2 + 
\bar{k}_{2 2}^1 (\bar{x}_2)^2 \big), 
\nonumber \\
\mathcal{K}_{x_2, \mathcal{\mathrm{QSSA}},\mathcal{X}_1,\mathcal{T}}(\{\bar{x}, \bar{y}\}; 
\, \bar{k}, \omega) & =   
\bar{k}_{0}^{2}  \omega^{-1} \bar{x}_2 \bar{y} + \bar{k}_{1}^2 \bar{x}_1 + 
\bar{k}_{2}^2  \bar{x}_2 + \bar{k}_{2 2}^2 (\bar{x}_2)^2, 
\nonumber\\
\mathcal{K}_{y, \mathcal{\mathrm{QSSA}},\mathcal{X}_1,\mathcal{T}}(\{\bar{x}_2, \bar{y} \}; 
\, \omega, \mu) & =  \mu^{-1} (\omega - \bar{x}_2 \bar{y}),
\label{QSSAhybrid} 
\end{align}
with $\bar{x}_2 (0) > 0$, $\omega > 0$, and $\mu \to 0$. Note that 
the chosen $\Psi_{\mathrm{QSSA}, \mathcal{X}_1, \mathcal{T}}$ does not 
introduce any additional fixed points when applied to system~(\ref{zerorot}).

\subsection{Step \emph{(3)}: construction of the canonical reaction network $\mathcal{R}_{\mathcal{K}^{-1}}$} \label{sec:step3}

Definition~\ref{def:cn} can be used to map the kinetic 
functions~(\ref{zerorot2}), (\ref{zerorot3}), (\ref{QSSAs}), 
and (\ref{QSSAhybrid}) to the canonical reaction 
networks $\mathcal{R}_{\mathcal{K}^{-1}}$. This is 
illustrated for (\ref{zerorot2}) in this section, and 
for (\ref{zerorot3}) and (\ref{QSSAhybrid}) in~\ref{AppSys}. 
For clarity, both the induced kinetic equations and the 
induced canonical reaction networks are presented. Note 
that the reaction networks are assumed to be taking place 
in an open reactor, and are not necessarily purely chemical 
in nature. Nevertheless, the non-chemical processes present 
in kinetic equations are represented as quasi-chemical reactions. 
Such reactions take form of input/output of chemical species, 
as well as containing quasi-species that are time-independent 
on a relevant time scale, so that their constant concentration 
is absorbed into a quasi-kinetics, leading to conservation 
laws not necessarily holding~\cite{Feinberg}. 

Writing $\mathbf{x} \equiv \mathbf{\bar{x}}$, the induced 
kinetic equations for~$(\ref{zerorot2})$ are given by
\begin{align}
\frac{\mathrm{d} x_1}{\mathrm{d} t}  & = 
k_1 x_1 + k_3 x_1^2 - k_5 x_1 x_2 - k_7 x_1^2 x_2 + k_8 x_1 x_2^2,  
\nonumber \\
\frac{\mathrm{d} x_2}{\mathrm{d} t} & = 
- k_2 x_2 + k_4 x_1 x_2 + k_6 x_2^2 - k_9 x_2^3,
\nonumber
\end{align}
while the induced canonical reaction network:
\begin{eqnarray*}
\begin{aligned}
& r_1: \;  & s_1  &
\xrightarrow[]{ k_1 } 2 s_1 ,  \; \; \; \; \; \; \; \; \; \; \; \; 
\; \; \; \; \; \; \; \; \; \; \; \; \; \; \;  \; \; \; \; r_2:  & 
s_2  &\xrightarrow[]{ k_2 } \varnothing,  \\
& r_3: \;  & 2 s_1  &\xrightarrow[]{ k_3 } 3 s_1, \; \; \; \; \;  
\; \; \; \; \; \; \; \; \; \; \; \; \;  \; \; \; \; \;  \; \; \; \; \; \; \; \;   
r_4: & s_1 + s_2  &\xrightarrow[]{ k_4 } s_1 + 2 s_2, \\
& r_5: \;  & s_1 + s_2  &\xrightarrow[]{ k_5 } s_2,  \; \; \; \; \; \; \; 
\; \; \; \; \; \; \; \; \; \; \; \; \; \; \; \; \; \; \; \; \;  \; \; \; \; 
\; r_6:  & 2 s_2  &\xrightarrow[]{ k_6 } 3 s_2, \\
& r_7: \;  & 2 s_1 + s_2  &\xrightarrow[]{ k_7 } s_1 + s_2, \; \; \; \; 
\; \; \; \; \; \; \; \;  \; \; \; \; \;  \; \; \; \; \; \; \; \;   
r_8: & s_1 + 2 s_2  &\xrightarrow[]{ k_8 } 2 s_1 + 2 s_2, \\
& r_9: \;  & 3 s_2  &\xrightarrow[]{ k_9 } 2 s_2,
\end{aligned}
\end{eqnarray*} 
where 
$k_1 = |\bar{k}_0^1|$, 
$k_2 = |\bar{k}_0^2|$, 
$k_3 = |\bar{k}_1^1|$, 
$k_4 = |\bar{k}_1^2|$, 
$k_5 = |\bar{k}_2^1|$, 
$k_6 = |\bar{k}_2^2|$, 
$k_7 = |\bar{k}_{1 2}^1|$, 
$k_8 = |\bar{k}_{2 2}^1|$, 
$k_9 = |\bar{k}_{2 2}^2|$, 
with the coefficients $\mathbf{\bar{k}}$ 
given by~(\ref{h5}), and the conditions given by~(\ref{Xonecond}). 

\section{Summary} 
\label{sec:end}
In the first part of the paper, a framework for 
constructing reaction systems having prescribed properties 
has been formulated as an inverse problem and presented in 
Section~\ref{sec:inverse}, relying on definitions introduced 
in Section~\ref{sec:notation}. As a part of the framework, 
in Section~\ref{sec:kt}, kinetic transformations have been 
defined that enable one to map an arbitrary polynomial ODE 
system with a set of constraints, possibly containing the 
cross-negative terms, into a kinetic one. Augmented with 
the results from~\cite{UNI1}, such transformations can 
be applied to nonpolynomial systems as well. Systems for 
which no affine transformation is kinetic have been defined 
as affinely nonkinetic in Section~\ref{sec:affine}, to 
emphasize the fact that significant changes to their 
solutions are required. X-factorable transformation~\cite{Xfactor}, 
that does not change the dimension of the systems being 
transformed, but introduces a higher number of nonlinear 
terms and leads to autocatalytic reaction networks, has 
been defined in Section~\ref{sec:xft}, and its properties 
when applied on two-dimensional systems have been derived 
in Theorem~\ref{Xfactdet}. The quasi-steady state transformation, 
that increases the dimension of the systems being transformed, 
but generally introduces a lower number of nonlinear terms, 
has been presented in Section~\ref{sec:QSSA}, and justified 
using Tikhonov's and Korzukhin's theorems~\cite{QSSA1}. 
As the focus of the paper has been more placed on the 
construction of kinetic equations, and less on constructions 
of reaction networks, in Section~\ref{sec:cn} an analytical 
and algorithmic method for obtaining the so-called canonical 
networks has been presented~\cite{Toth2}. The framework 
may be used for constructing lower-dimensional reaction systems 
displaying exotic phenomena, with Algorithm~\ref{tabinv} 
exemplifying the construction process. 

In the second part of the paper, the inverse problem framework 
has been applied to a case study of constructing bistable 
reaction systems undergoing a supercritical homoclinic bifurcation, 
with a parameter controlling the stable sets separation, with 
an overview of the procedure presented in Section~\ref{sec:formulation}. 
In Section~\ref{sec:step1}, a polynomial ODE system having a 
homoclinic orbit in the phase plane has been constructed 
using the results from~\cite{Sand}, and perturbed in 
such a way that the sufficient conditions for the existence 
of the homoclinic bifurcation are fulfilled, as required 
by Andronov-Leontovich Theorem~\cite{Bifur1}. 
In Section~\ref{sec:step2}, the kinetic transformations from 
Section~\ref{sec:inverse} have been used in order to map the 
polynomial system to a kinetic one with the homoclinic orbit 
in the positive quadrant. The topological phase space effects 
produced by the kinetic transformations on the constructed systems 
have been discussed. In Section~\ref{sec:step3} and~\ref{AppSys}, 
the canonical reaction networks induced by some of the kinetic 
equations have been presented. 

In this paper, we have constructed chemical reaction networks 
inducing two-dimensional cubic kinetic functions with the 
deterministic ODEs (kinetic equations) undergoing a supercritical 
homoclinic bifurcation. In a future publication, we will 
report our results on the stochastic analysis of the constructed 
systems, consisting of examining the quasi-stability of 
the limit cycle, and stochastic switching between the stable 
sets, as a function of the bifurcation parameter and the 
parameter controlling the stable set separation. A motivation 
for such a study is the fact that stochastic effects play 
an important role in systems biology due to the inherently 
small reactors~\cite{Intro1,Intro2,Intro3,Intro4}, and 
one might even say that systems biology has initiated 
a renaissance of the stochastic reaction kinetics~\cite{StochReac}. 

\appendix

\section{Oscillations in two-component bimolecular reaction 
systems with cross-negative terms}
\label{OsciCNT}

\noindent
{\bf Theorem A.1:}
{\it Consider a two-dimensional ODE system with a quadratic polynomial 
RHS, 
$\boldsymbol{\mathcal{P}}(\mathbf{x}; \, \mathbf{k}) \in \mathbb{P}_{2}(\mathbb{R}^{2}; \, \mathbb{R}^{2})$, $\mathbf{k} \in \mathbb{R}^{12}$:
\begin{align}
\frac{\mathrm{d} x_1}{\mathrm{d} t} = 
\; \; \mathcal{P}_1(\mathbf{x}; \, \mathbf{k})  & = k_{0}^{1} + k_{1}^1 x_1 + k_{2}^1 x_2 + k_{1 1}^1 x_1^2 + k_{1 2}^1 x_1 x_2 + k_{2 2}^1 x_2^2, 
\nonumber \\
\frac{\mathrm{d} x_2}{\mathrm{d} t} = \; \; \mathcal{P}_2(\mathbf{x}; \, \mathbf{k}) & =  
k_{0}^{2} + k_{1}^2 x_1 + k_{2}^2 x_2 
+ k_{1 1}^2 x_1^2 + k_{1 2}^2 x_1 x_2 + k_{2 2}^2 x_2^2, \label{positivity}
\end{align}
If $k_{1 1}^1 \le 0$ and $k_{2 2}^2 \le 0$, and if system~$(\ref{positivity})$ is \emph{nonnegative}, then the system has no limit cycles.}

\begin{proof}
Considering $x_1, x_2 > 0$, writing 
$\boldsymbol{\mathcal{P}}(\mathbf{x}; \, \mathbf{k}) 
= \big(\mathcal{P}_1(x_1,x_2; \, \mathbf{k}), \mathcal{P}_2(x_1,x_2; \, \mathbf{k}) \big)$, 
and letting the Dulac function to be given by $B(x_1,x_2) = (x_1 x_2)^{-1}$, 
it follows that
\begin{align}
D(x_1,x_2; \, \mathbf{k}) & = 
\frac{\partial}{\partial x_1} \big( B(x_1,x_2) \mathcal{P}_1(x_1,x_2; \, \mathbf{k}) \big)  
+ \frac{\partial}{\partial x_2} \big( B(x_1,x_2) \mathcal{P}_2(x_1,x_2; \, \mathbf{k}) \big) \nonumber \\
& = 
- \left(\frac{k_{0}^1}{x_1^2 x_2} + \frac{k_{2}^1}{x_1^2} + \frac{k_{2 2}^1 x_2}{x_1^2} \right) 
+ \frac{k_{1 1}^1}{x_2} 
- \left(\frac{k_{0}^2}{x_1 x_2^2} + \frac{k_{1}^2}{x_2^2} + \frac{k_{1 1}^2 x_1}{x_2^2} \right) 
+ \frac{k_{2 2}^2}{x_1}. \nonumber
\end{align}
Multiplying by $- (x_1 x_2)^2$, and defining a new function 
$\bar{D}(x_1,x_2; \, \mathbf{k}) = - (x_1 x_2)^2 D(x_1,x_2; \, \mathbf{k})$, results in
\begin{equation*}
\bar{D}(x_1,x_2; \, \mathbf{k}) = 
x_2 \left[\mathcal{P}_1(0,x_2) - k_{1 1}^1 x_1^2 \right] 
+ x_1 \left[ \mathcal{P}_2(x_1,0) - k_{2 2}^2 x_2^2 \right].
\end{equation*}
If $k_{1 1}^1 \le 0$ and $k_{2 2}^2 \le 0$, and if system~(\ref{positivity}) is nonnegative, so that $\mathcal{P}_1(0,x_2) \ge 0$ and $\mathcal{P}_2(x_1,0) \ge 0$, then $\bar{D} \ge 0$. The only case when a limit cycle may exist is if $\bar{D} = 0$ for all $x_1, x_2 >0$, and in~\cite{Constr1} it is shown that no limit cycles exist in this case.
\end{proof}
\noindent
In~\cite{Constr1}, it was shown that the absence of cross-negative terms in $\boldsymbol{\mathcal{P}}(\mathbf{x}; \, \mathbf{k}) \in \mathbb{P}_{2}(\mathbb{R}^{2}; \, \mathbb{R}^{2})$, with $k_{1 1}^1 \le 0$ and $k_{2 2}^2 \le 0$, implies the absence of limit cycles, i.e. one requires the more restrictive condition $\boldsymbol{\mathcal{P}}(\mathbf{x}; \, \mathbf{k}) \in \mathbb{P}^{\mathcal{K}}_{2}(\mathbb{R}_{\ge}^{2}; \, \mathbb{R}^{2})$. Theorem A.1 shows that, in fact, absence of the cross-negative effect in $\boldsymbol{\mathcal{P}}(\mathbf{x}; \, \mathbf{k}) \in \mathbb{P}_{2}(\mathbb{R}^{2}; \, \mathbb{R}^{2})$, with $k_{1 1}^1 \le 0$ and $k_{2 2}^2 \le 0$, implies the absence of limit cycles, i.e. one requires the less restrictive condition $\boldsymbol{\mathcal{P}}(\mathbf{x}; \, \mathbf{k}) \in \mathbb{P}_{2}(\mathbb{R}_{\ge}^{2}; \, \mathbb{R}^{2})$.

\section{Andronov-Leontovich theorem}
\label{HomoclinicTheorem}

\noindent
{\bf Andronov-Leontovich theorem}~\cite{Bifur1}:
{\it Consider system~$(\ref{eqn:nkv})$ with $\boldsymbol{\mathcal{P}}(\mathbf{x}; \, \mathbf{k},\alpha) \in \mathbb{P}_{m}(\mathbb{R}^2; \, \mathbb{R}^2)$, $m \in \mathbb{N}$, where $\alpha \in \mathbb{R}$ is a bifurcation parameter. Let $\lambda_1(\alpha)$ and $\lambda_2(\alpha)$ be eigenvalues of the Jacobian corresponding to the two-dimensional system~$(\ref{eqn:nkv})$, $J = \nabla \boldsymbol{\mathcal{P}}(\mathbf{x}; \, \mathbf{k},\alpha)$, and suppose that at $\alpha = 0$, the following \emph{homoclinic bifurcation conditions} {\rm (i)--(ii)} are satisfied, and that~$(\ref{eqn:nkv})$ is \emph{generic}, i.e. the following \emph{homoclinic genericity conditions} {\rm (iii)--(iv)} are satisfied: 
\begin{romanlist}[(i)]
\item{System has a saddle fixed point $\mathbf{x}^{*} = 0$ with eigenvalues $\lambda_1(0) < 0 < \lambda_2(0)$.}
\item{System has a homoclinic orbit $\gamma^{*}$ to the saddle fixed point $\mathbf{x}^{*}$.}
\item{Nondegeneracity condition:} $\sigma_0 =\lambda_1(0) + \lambda_2(0) \ne 0$, where $\sigma_0 \in \mathbb{R}$ is called the \emph{saddle quantity}. 
\item{Transversality condition:} \emph{Melnikov integral}, $M(\alpha)$, along the homoclinic orbit satisfies:
\begin{align}
M(0) & =  \int_{- \infty}^{+ \infty} {\varphi(t)} \left( \boldsymbol{\mathcal{P}}(\mathbf{x}; \, \mathbf{k},\alpha) \times \frac{\partial  \boldsymbol{\mathcal{P}}(\mathbf{x}; \, \mathbf{k},\alpha)}{\partial \alpha} \right) \mathrm{d} t  \; \; \ne 0, \label{Melnikov}
\end{align}
where $\varphi(t) = \exp \left(- \int_{0}^{t} (\nabla \cdot \boldsymbol{\mathcal{P}}(\mathbf{x}; \, \mathbf{k},\alpha) d \tau \right)$, $\varphi(t) > 0$.  This is equivalent to splitting of the homoclinic orbit at the bifurcation with a nonzero speed.
\end{romanlist}
Then, for all sufficiently small $|\alpha|$, there exists a neighbourhood of the saddle fixed point and the homoclinic orbit such that a unique limit cycle bifurcates from $\gamma^{*}$. If $\sigma_0 < 0$, the homoclinic bifurcation is \emph{supercritical}, giving rise to a unique \emph{stable} limit cycle, while if $\sigma_0 > 0$, the homoclinic bifurcation is \emph{subcritical}, giving rise to a unique \emph{unstable} limit cycle.
}

\section{The canonical reaction networks induced 
by~$(\ref{zerorot3})$ and~$(\ref{QSSAhybrid})$}
\label{AppSys}

Writing $\mathbf{x} \equiv \mathbf{\bar{x}}$, 
the induced kinetic equations for~$(\ref{zerorot3})$ are given by
\begin{align}
\frac{\mathrm{d} x_1}{\mathrm{d} t}  & = -k_1 x_1 - k_3 x_1^2 + k_5 x_1 x_2 - k_7 x_1^3 + k_9 x_1^2 x_2 - k_{11} x_1 x_2^2, \nonumber \\
\frac{\mathrm{d} x_2}{\mathrm{d} t} & = k_2 x_2 + k_4 x_1 x_2 - k_6 x_2^2 - k_8 x_1^2 x_2 + k_{10} x_1 x_2^2 - k_{12} x_2^3, \nonumber
\end{align}
while the canonical reaction network:
\begin{eqnarray}
\begin{aligned}
& r_1: \;  & s_1  &\xrightarrow[]{ k_1 } \varnothing,  \; \; \; \; \; \; \; \;  \; \; \; \; \; \; \; \; \; \; \; \; \; \; \; \; \; \;  \; \; \; \; \;   r_2:  & s_2  &\xrightarrow[]{ k_2 } 2 s_2, \nonumber \\
& r_3: \;  & 2 s_1  &\xrightarrow[]{ k_3 } s_1, \; \; \; \; \;  \; \; \; \; \; \; \; \; \; \; \; \; \;  \; \; \; \; \;  \; \; \; \; \; \; \; \;   r_4: & s_1 + s_2  &\xrightarrow[]{ k_4 } s_1 + s_2 + \mathrm{sign}(\bar{k}_1^2) s_2, \\
& r_5: \;  & s_1 + s_2  &\xrightarrow[]{ k_5 } 2 s_1 + s_2,  \; \; \; \; \; \; \; \;  \; \; \; \; \;  \; \; \;  \; \; \; \; \; r_6:  & 2 s_2  &\xrightarrow[]{ k_6 } s_2, \\
& r_7: \;  & 3 s_1 &\xrightarrow[]{ k_7 } 2 s_1, \; \; \; \; \;  \; \; \;  \; \; \; \; \; \; \; \;  \; \; \; \; \;  \; \; \; \; \; \; \; \;   r_8: & 2 s_1 + s_2  &\xrightarrow[]{ k_8 } 2 s_1, \\
& r_9: \;  & 2 s_1 + s_2  &\xrightarrow[]{ k_9 } 3 s_1 + s_2,  \; \; \; \; \; \; \;  \; \; \; \; \;  \; \; \; \; \; \; \; \;   r_{10}: & s_1 + 2 s_2  &\xrightarrow[]{ k_{10} } s_1 + 3 s_2,\\
& r_{11}: \;  & s_1 + 2 s_2  &\xrightarrow[]{ k_{11} }2 s_2, \; \; \; \; \; \; \; \; \; \; \; \; \;  \; \; \; \; \;  \; \; \; \; \; \; \; \;   r_{12}: & 3 s_2  &\xrightarrow[]{ k_{12} } 2 s_2,
\end{aligned}
\end{eqnarray} 
where $k_1 = |\bar{k}_0^1|$, $k_2 = |\bar{k}_0^2|$, $k_3 = |\bar{k}_1^1|$, $k_4 = |\bar{k}_1^2|$, $k_5 = |\bar{k}_2^1|$, $k_6 = |\bar{k}_2^2|$, $k_7 = |\bar{k}_{1 1}^1|$, $k_8 = |\bar{k}_{1 1}^2|$, $k_9 = |\bar{k}_{1 2}^1|$, $k_{10} = |\bar{k}_{1 2}^2|$, $k_{11} = |\bar{k}_{2 2}^1|$, $k_{12} = |\bar{k}_{2 2}^2|$, with the coefficients $\mathbf{\bar{k}}$ given by~(\ref{x2coeff}), and the conditions given by~(\ref{conditions3}). Note that by taking $a \in \big(-\frac{8}{9}, \frac{1}{5} (2 - \sqrt{34}) \big)$ and $\mathcal{T} = -\frac{2}{3} (2 + 3 a)^{-1}$, it follows that $k_4 = 0$.

\medskip

\noindent
Writing $\mathbf{x} \equiv \mathbf{\bar{x}}$, the induced kinetic equations for~$(\ref{QSSAhybrid})$ are given by
\begin{align}
\frac{\mathrm{d} x_1}{\mathrm{d} t}  & = k_1 x_1 + k_3 x_1^2 - k_5 x_1 x_2 - k_7 x_1^2 x_2 + k_8 x_1 x_2^2, \nonumber \\
\frac{\mathrm{d} x_2}{\mathrm{d} t} & = - k_2 x_2 x_3 + k_4 x_1 + k_6 x_2 - k_9 x_2^2, \nonumber \\
\frac{\mathrm{d} x_3}{\mathrm{d} t} & = k_{10} - k_{11} x_2 x_3, \nonumber 
\end{align}
while the canonical reaction network:
\begin{eqnarray*}
\begin{aligned}
& r_1: \;  & s_1  &\xrightarrow[]{ k_1 } s_1 + \mathrm{sign}(\bar{k}_{0}^1) s_1,  \; \; \; \; \; \; \; \;  \; \; \; \; r_2:  & s_2 + s_3 &\xrightarrow[]{ k_2 } s_3, \\
& r_3: \;  & 2 s_1  &\xrightarrow[]{ k_3 } 3 s_1, \; \; \; \; \;  \; \; \; \; \; \; \; \; \; \; \; \; \;  \; \; \; \; \;  \; \; \; \; \; \; \; \;   r_4: & s_1  &\xrightarrow[]{ k_4 } s_1 + s_2, \\
& r_5: \;  & s_1 + s_2  &\xrightarrow[]{ k_5 } s_2,  \; \; \; \; \; \; \; \; \; \; \; \; \; \; \; \; \; \; \; \; \; \; \; \; \; \; \; \; \; \; \; \; \; r_6:  & s_2  &\xrightarrow[]{ k_6 } 2 s_2, \\
& r_7: \;  & 2 s_1 + s_2  &\xrightarrow[]{ k_7 } s_1 + s_2, \; \; \; \; \; \; \; \; \; \; \; \;  \; \; \; \; \;  \; \; \; \; \; \; \; \;   r_8: & s_1 + 2 s_2  &\xrightarrow[]{ k_8 } 2 s_1 + 2 s_2, \\
& r_9: \;  & 2 s_2  &\xrightarrow[]{ k_9 } s_2, \; \; \; \; \; \; \; \; \; \;\; \;\; \; \; \; \; \; \; \; \; \; \;  \; \; \; \; \; \; \; \;   r_{10}: & \varnothing &\xrightarrow[]{ k_{10} } s_3, \\
 & r_{11}: & s_2 + s_3 &\xrightarrow[]{ k_{11} } s_2,
\end{aligned}
\end{eqnarray*} 
with 
$k_1 = |\bar{k}_0^1|$, 
$k_2 = |\bar{k}_0^2|$, 
$k_3 = |\bar{k}_1^1|$, 
$k_4 = |\bar{k}_1^2|$, 
$k_5 = |\bar{k}_2^1|$, 
$k_6 = |\bar{k}_2^2|$, 
$k_7 = |\bar{k}_{1 2}^1|$, 
$k_8 = |\bar{k}_{2 2}^1|$, 
$k_9 = |\bar{k}_{2 2}^2|$, 
$k_{10} = \mu \omega$, 
$k_{11} = \mu$, the coefficients 
$\mathbf{\bar{k}}$ given by~(\ref{h5}), 
$\omega > 0$, $\mu \to 0$, and conditions given 
by~(\ref{Xonecond}), with the lower bound on $\mathcal{T}_2$ 
being $1$. Note that by taking 
$\mathcal{T}_1 = - \mathcal{T}_2 a^{-1}$, it follows that $k_1 = 0$.

\bigskip

\noindent
{\bf Acknowledgments:}
The research leading to these results has received funding from the European 
Research Council under the \textit{European Community}'s Seventh Framework 
Programme ({\it FP7/2007-2013})/ERC {\it grant agreement} n$^o$  239870,
and from the People Programme (Marie Curie Actions) of the European Union's 
Seventh Framework Programme ({\it FP7/2007-2013}) under REA grant agreement no. 328008.
Tom\'a\v s Vejchodsk\'y gratefully acknowledges the support of RVO 67985840. 
Radek Erban would like to thank the Royal Society for a University Research 
Fellowship and the Leverhulme Trust for a Philip Leverhulme Prize.

\end{document}